\documentclass[11pt,reqno]{amsproc}

\usepackage{amsmath,amsfonts,amssymb,amsthm}
\usepackage[abbrev,lite,nobysame]{amsrefs}
\usepackage{mathrsfs,esint}
 \usepackage[usenames,dvipsnames]{color}
\usepackage[margin=1in]{geometry}
\usepackage{dsfont}
\usepackage{enumerate,enumitem}
\usepackage{comment}
\usepackage{bbm}
\usepackage{graphicx}

\usepackage[colorlinks=true, pdfstartview=FitV, linkcolor=BrickRed,citecolor=black, urlcolor=black]{hyperref}

\usepackage{mathtools}
\mathtoolsset{showonlyrefs}

\newtheorem{MAINtheorem}{Theorem}
\newtheorem{proposition}{Proposition}[section]
\newtheorem{lemma}[proposition]{Lemma}

\newtheorem{theorem}[proposition]{Theorem}

\theoremstyle{definition}
\newtheorem{definition}[proposition]{Definition}

\newtheorem{example}[proposition]{Example} 
\newtheorem{remark}[proposition]{Remark}

\numberwithin{equation}{section}

 \makeatletter

\renewcommand\subsubsection{\@startsection{subsubsection}{3}%
\normalparindent{.5\linespacing\@plus.7\linespacing}{-.5em}
{\normalfont\bfseries}}

\def\@tocline#1#2#3#4#5#6#7{\relax
  \ifnum #1>\c@tocdepth 
  \else
    \par \addpenalty\@secpenalty\addvspace{#2}%
    \begingroup \hyphenpenalty\@M
    \@ifempty{#4}{%
      \@tempdima\csname r@tocindent\number#1\endcsname\relax
    }{%
      \@tempdima#4\relax
    }%
    \parindent\z@ \leftskip#3\relax \advance\leftskip\@tempdima\relax
    \rightskip\@pnumwidth plus4em \parfillskip-\@pnumwidth
    #5\leavevmode\hskip-\@tempdima
      \ifcase #1
       \or\or \hskip 1em \or \hskip 2em \else \hskip 3em \fi%
      #6\nobreak\relax
    \dotfill\hbox to\@pnumwidth{\@tocpagenum{#7}}\par
    \nobreak
    \endgroup
  \fi}
\makeatother

\newcommand\eps{\varepsilon}
\newcommand\e{{\rm e}}
\newcommand\dd{{\rm d}}
\newcommand\Ker{{\rm Ker}}
\newcommand\Ran{{\rm Ran}}
\newcommand\Id{{\rm Id}}
\def\Re{{\rm Re}}
\newcommand{\dist}{{\rm dist}}

\newcommand{\RR}{\mathbb{R}}
\newcommand{\NN}{\mathbb{N}}
\newcommand{\CC}{\mathbb{C}}
\newcommand{\TT}{\mathbb{T}}
\newcommand{\ZZ}{\mathbb{Z}}
\newcommand{\JJ}{\mathbb{J}}
\newcommand{\jj}{\boldsymbol{j}}
\newcommand{\cI}{\mathcal{I}}
\newcommand{\cL}{\mathcal{L}}

\newcommand{\cS}{\mathcal{S}}
\newcommand{\fU}{\mathfrak{U}}
\newcommand{\bk}{\boldsymbol{k}}
\newcommand{\ini}{{\rm in}}
\newcommand{\supp}{{\rm supp}}

\newcommand\Uabc{U_{\mathsf a \mathsf b \mathsf c}}
\newcommand\hUabc{\widehat{U}_{\mathsf a \mathsf b \mathsf c}}
\newcommand\sfa{\mathsf{a}}
\newcommand\sfb{\mathsf{b}}
\newcommand\sfc{\mathsf{c}}

\begin{document}
\title[Alpha-unstable flows and the fast dynamo problem]{\vspace*{-3cm}Alpha-unstable flows and the fast dynamo problem} 

\author[M. Coti Zelati]{Michele Coti Zelati}
\address{Department of Mathematics, Imperial College London}
\email{m.coti-zelati@imperial.ac.uk}

\author[M. Sorella]{Massimo Sorella}
\email{m.sorella@imperial.ac.uk}

\author[D. Villringer]{David Villringer}
\email{d.villringer22@imperial.ac.uk}



\begin{abstract}
We construct a time-independent, incompressible, and Lipschitz-continuous velocity field in $\mathbb{R}^3$ that generates a fast kinematic dynamo - an instability characterized by exponential growth of magnetic energy, independent of diffusivity. Specifically, we show that the associated vector transport-diffusion equation admits solutions that grow exponentially fast, uniformly in the vanishing diffusivity limit $\varepsilon\to0$. Our construction is based on a periodic velocity field $U$ on $\mathbb{T}^3$, such as an Arnold–Beltrami–Childress flow, which satisfies a generic spectral instability property called alpha-instability, established via perturbation theory. This provides a rigorous mathematical framework for the alpha-effect, a mechanism conjectured in the late 1960s to drive large-scale magnetic field generation. By rescaling with respect to $\varepsilon$ and employing a Bloch-type theorem, we extend the solution to the whole space. Finally, through a gluing procedure that spatially localizes the instability, we construct a globally defined velocity field $u$ in $\mathbb{R}^3$
that drives the dynamo instability.
\end{abstract}


\maketitle

\tableofcontents

\section{The fast dynamo problem}\label{sec:intro}

The generation and maintenance of magnetic fields in astrophysical and geophysical settings are often attributed to dynamo action, where fluid motion amplifies magnetic fields through the process of electromagnetic induction. At the heart of this phenomenon lies the \emph{kinematic dynamo equation}, derived from Maxwell’s equations and Ohm’s law in a  non-relativistic    moving conductor, which governs the evolution of a divergence-free magnetic field   $B^\eps:[0,\infty)\times\RR^3\to\RR^3$ and is written as
\begin{equation} \label{passive-vector}\tag{KDE}
\begin{cases}
\partial_t B^\eps=\nabla \times (u\times B^\eps)+\eps \Delta B^\eps ,\\
\nabla \cdot B^\eps =0,\\
B^\eps|_{t=0} = B^\eps_{\ini},
\end{cases}
\qquad (t,x)\in (0,\infty)\times \RR^3.
\end{equation}
In this work, $u:\RR^3\to\RR^3$ is a given, divergence-free and time-independent velocity field, and $\eps>0$ is the magnetic diffusivity, inversely proportional to the magnetic Reynolds number.

A central challenge in this field is understanding the mechanisms behind fast dynamos - flows capable of sustaining magnetic field growth at rates independent of magnetic diffusivity. Mathematically, this corresponds to an $\eps$-independent exponential growth of the $L^2$ norm of $B^\eps$, also known as the total magnetic energy. In precise mathematical terms, a bounded Lipschitz continuous velocity field $u$ is a \emph{kinematic dynamo} on a domain $\mathcal{M}\subset \RR^3$ if for any $\eps >0$ there exists  $B_\ini^\eps \in L^2 (\mathcal{M})$ such that the corresponding solution  of \eqref{passive-vector}, endowed with physical boundary condition if $\partial\mathcal{M}\neq \emptyset$,   satisfies
$$ 
\gamma_\eps:= \limsup_{t\to\infty}\frac1t \log \| B^\eps (t) \|_{L^2(\mathcal{M})}>0.
$$
If $u$ is a kinematic dynamo and $\displaystyle\liminf_{\eps\to0} \gamma_\eps >0$, then $u$ is a \emph{fast} dynamo, otherwise it is said to be a \emph{slow} dynamo.
The question, originally posed by Ya.B. Zeldovich and A.D. Sakharov in the 1970s, asks whether there exists a divergence-free velocity field in $\mathcal{M}$ that is a fast dynamo  (see \cite{AK98}*{Chapter V}  and \cite{Arnold04}*{Pb. 1994-28}).   We resolve this problem in the case $\mathcal{M}=\RR^3$.
\begin{MAINtheorem} \label{thm:main}
There exist an autonomous, divergence-free velocity field $u \in W^{1, \infty} (\RR^3)$ and constants $\gamma, c_0>0$  with the following property. For any $\eps \in (0,1)$ there exists a non-zero, divergence-free initial datum $B^\eps_{\ini}  \in L^2 (\RR^3)$ such that the corresponding solution 
of \eqref{passive-vector}  satisfies 
$B^\eps\in L^\infty_{loc}(0,\infty;L^2(\RR^3))\cap L^2_{loc}(0,\infty;\dot{H}^1(\RR^3))$ and 
\begin{equation}\label{eq:dynamoaction}
    \| B^\eps (t) \|_{L^2(\RR^3)} \geq c_0 \e^{\gamma t} \| B^\eps_{\ini}  \|_{L^2 (\RR^3)},
\end{equation} 
for every $t\geq 0$.
\end{MAINtheorem}

This result provides an example of \emph{fast dynamo action} by a Lipschitz, time-independent velocity field on the whole space $\RR^3$. The exponential growth \eqref{eq:dynamoaction} is the fastest possible that a finite-energy initial datum can experience under the dynamics of \eqref{passive-vector} with a Lipschitz velocity field.

While we outline the proof of Theorem \ref{thm:main} in Section \ref{sub:strategy}, we mention that the construction of $u$ is based on the concept of \emph{alpha-unstable flows} on the periodic domain $\TT^3$, defined as follows.

\begin{definition}[Alpha-unstable flow] \label{def:alpha}
Let $U \in C^\infty (\TT^3)$ be a mean-free, incompressible velocity field  with $\|\nabla U\|_{L^\infty}<1$. For $v\in \CC^3$, consider the solution map $\cS:\CC^3\to L^2(\TT^3)$, $v\mapsto \cS(v)= S$ of the elliptic problem
\begin{equation} \label{eq:ellipticKernel}
            \nabla \times (U \times S)+\Delta  S =  \nabla \times ( v \times U ),\qquad 
            \int_{\TT^3} S(x)\dd x =0,
    \end{equation}
and define the $\CC^{3\times 3}$ matrix $A=A(U,\jj)$ via
\begin{equation}\label{eq:Amatrix}
Av=i \frac{\jj}{|\jj|} \times \fint_{\TT^3}  U \times \cS(v),\qquad \jj\in \RR^3\setminus\{0\}, \qquad v\in \CC^3.
\end{equation}
We say that $U$ is \emph{alpha-unstable} if there exist a nonzero $\jj\in \RR^3$ and a simple eigenvalue $\mu \in \CC$ of $A$ with $\Re (\mu) >0$.
\end{definition}

The assumption $\| \nabla U \|_{L^\infty} < 1$ in Definition \ref{def:alpha} is only needed to ensure existence and uniqueness of solutions to equation \eqref{eq:ellipticKernel}, so that $\cS$ is well-defined. 
Although the matrix in \eqref{eq:Amatrix} may seem obscure at first, it arises naturally when analyzing the spectrum of the linear operator governing the dynamics of \eqref{passive-vector}, as well as its perturbations. 
It also connects with the so-called alpha-effect, as we explain in the next Section \ref{sub:alpha}. To the best of our knowledge, this is the first rigorous justification of this effect for \eqref{passive-vector}.

Alpha-unstable flows can be proved to be generic in the space of smooth, mean-free and incompressible velocities (c.f. Proposition \ref{prop:alphaDensity}). However, their existence does not immediately imply fast dynamo action on $\TT^3$. This remains to this date an outstanding open problem, see \cite{Arnold04}*{Pb. 1994-28}. 

\subsection{Fast dynamos, chaotic flows and anti-dynamos}
The idea that fluid motion in a conducting medium could amplify magnetic fields dates back to Larmor’s 1919 work \cite{larmor1919possible}, which proposed dynamo action as the mechanism behind the Sun’s magnetic field. However, a rigorous mathematical framework for dynamos only emerged in the mid-20th century.

The full dynamo problem is governed by the nonlinear magnetohydrodynamics (MHD) equations, where \eqref{passive-vector} is coupled to the Navier-Stokes system describing fluid motion. In this setting, exponential magnetic energy growth cannot persist indefinitely, as it would violate energy bounds. Instead, one seeks a mechanism for finite-time transfer from kinetic to magnetic energy. While extensive studies exist in applied settings (see \cites{CG95,AK98,brandenburg2005astrophysical,Moffatt78,rudiger2006magnetic}), only a few mathematical works have addressed this problem \cites{FV91,Gerardvaret05,GR07,Vishik86}.

Neglecting the Lorentz force’s feedback on the velocity field leads to the kinematic dynamo equation \eqref{passive-vector}, an approximation valid when the magnetic field is weak - such as in the initial phase of seed field amplification driven by differential rotation. In this setting, searching for exponentially growing solutions is both feasible and an outstanding challenge.

A remarkable result in this direction is provided by Vishik \cite{Vishik89}, who showed that fast dynamo action implies that the (non-diffusive) Lagrangian top Lyapunov exponent of $u$ is positive. Chaotic velocity fields can be constructed using a random dynamical systems approach \cites{CZNF24,BBPS22}. In particular, \cite{CZNF24} demonstrates that a simple randomization of ABC flows yields a universal \emph{ideal} dynamo - i.e., a velocity field for which any nontrivial $L^2$ initial datum leads to an exponentially growing solution of \eqref{passive-vector} in the case $\eps=0$.
However, in general, treating the case $\eps>0$ as a perturbation of $\eps=0$ is a highly non-trivial matter. The Laplacian is singular with respect to the vector-transport operator ($\eps=0$), rendering the spectral theory outlined in Appendix \ref{appendix} inapplicable. Moreover, existing techniques for establishing spectral continuity under singular perturbations - such as those employed in \cite{Albritton_Brué_Colombo_2022} - require the unperturbed operator to have a non-empty discrete spectrum. Yet, as shown in \cite{Chicone_Latushkin_Montgomery-Smith_1995}, the ideal dynamo operator has a spectrum consisting of a single vertical strip, implying that its discrete spectrum is empty. Thus, whether the existence of an ideal dynamo necessarily implies the existence of a fast dynamo remains unclear. In fact, results such as those in \cite{Soward_1987} suggest that this implication fails in general.

Another key difficulty in dynamo theory is its resistance to simplification. Throughout the 20th century, various \emph{anti-dynamo theorems} ruled out dynamo action for specific velocity and magnetic field configurations. Notably, Zeldovich’s theorem \cites{zeldovich1980magnetic,Zeldovich_1992} states that a velocity field with zero vertical component cannot sustain a dynamo, while Cowling’s theorem \cite{Cowling33} rules out axisymmetric solutions. These results highlight that a functioning dynamo requires a genuinely three-dimensional magnetic field \cite{AK98}. The only examples of fast dynamo on $\RR^3$ that we are aware of are the following:

\begin{itemize}
    \item In \cite{ZRMS84}, with a velocity field of the form $u(t,x)=C(t)x$, where $C(t)$ is a traceless random matrix. This construction relies heavily on the \emph{unboundedness} of the velocity field, since it requires the support of the initial datum to grow exponentially. This cannot happen with a bounded velocity field as in our Theorem \ref{thm:main};
    \item In \cite{Gilbert88}, with the so-called Ponomarenko model. This is a \emph{discontinuous} helical flow, for which exponential growth happens in a neighborhood of the discontinuity.
\end{itemize}

\subsection{The alpha-effect}\label{sub:alpha}
The concept of \emph{alpha-instability} in Definition \ref{def:alpha} is named after classical works that sought to exploit the so-called \emph{alpha-effect} - a mechanism introduced to circumvent anti-dynamo theorems while avoiding excessive complexity. The most relevant work for our study is Roberts’ 1970 analysis of dynamo action in periodic velocity fields  \cite{Roberts70}, later refined in \cite{Roberts72}. Roberts considered a periodic magnetic field $B(t,x_1,x_2)$ that, by itself, cannot generate a dynamo due to the anti-dynamo theorems. However, introducing a slow variation in the $x_3$-direction allows the system to escape this constraint, with the alpha-effect determining the extent of dynamo action. While highly effective, Roberts’ approach did not yield quantitative bounds on the dynamo growth rate in terms of the magnetic resistivity 
$\eps$. In particular, it remained unclear whether dynamo action persisted in the limit 
$\eps\to0$, see further studies \cites{Childress_1979,Soward_1987}.

The importance of the alpha-effect in the dynamo problem can be at justified heuristically following \cites{Roberts70,Roberts72,Moffatt78,Childress_1979,CG95}. The crucial mechanism underlying the alpha-effect is the interaction between small and large scales. Specifically, if the velocity field $U$ is concentrated at high frequencies (i.e., small scales) while the initial magnetic field is only at low frequencies (i.e., large scales), the term $\nabla \times (U \times B^\eps )$ in \eqref{passive-vector} induces the creation of high frequencies in the magnetic field. 

To clarify this mechanism, we introduce the two-scale approach in the context of the passive vector equation, first developed by Steenbeck, Krause, and Rädler \cites{SK69,SKR66}. We decompose the magnetic field into small and large scales components: 
\begin{equation}\label{eq:BFBS}
    B^\eps = B_S + B_L, \qquad 
    U = U_S \,,
\end{equation}
where $B_L$ represents the large-scale part, and $B_S$ represents the small-scale part. A common homogenization technique is to assume that $B^\eps$ has the two-scale structure $B^\eps(x)= B_L (x) + B_S (x, \lambda x )$, where $y = \lambda x$ is the small scales variable for $\lambda \gg 1$. Using this approach, the passive vector equation decomposes into separate equations as
\begin{align}
\begin{cases}
    \partial_t B_L = \nabla \times  \langle U_S \times B_S \rangle_y + \eps \Delta B_L \,, \label{eq:slow-dynamo}
    \\
    \partial_t B_S = \nabla \times (U_S \times B_L) + \nabla \times (U_S \times B_S - \langle U_S \times B_S \rangle_y) + \eps \Delta B_S \,,
\end{cases}
\end{align}
where $\langle \cdot  \rangle_y$ 
denotes averaging in the small scales.

The key term here is the mean electromotive force $\mathcal{E} =  \langle U_S \times B_S \rangle_y$, which plays a central role in dynamo theory by potentially sustaining the large scales part $B_L$. A crucial insight is that the growth of $B_L$ and $B_S$ reinforce each other \cite{Moffatt78}. In particular, it is predicted that the relation 
$$ 
\langle  U_S \times B_S \rangle_y \approx  \alpha (B_L)
$$
holds for some matrix $\alpha : \RR^3 \to \RR^3$, known as the \emph{alpha-matrix}. This is the so-called alpha-effect, from which the matrix derives its name.

If this analysis is correct, the first equation in \eqref{eq:slow-dynamo} on the Fourier side $x\mapsto\bk$ reads
$$
\partial_t \widehat{B}_L = i\bk \times  \alpha( \widehat{B}_L)  - \eps|\bk|^2 \Delta \widehat{B}_L.
$$
If the matrix $i  \bk  \times (\alpha (\cdot))$ has an eigenvalue $\mu$ with a positive real part for some
$\bk\in \RR^3\setminus\{0\}$, then the solution with corresponding eigenvector as initial datum exhibits exponential growth for sufficiently small 
$\eps$.
This heuristic argument, originating from \cites{Roberts70,Roberts72,Soward_1987,CG95,Childress_1979}, provides an important qualitative prediction about dynamo behavior. However, to the best of our knowledge, there are no rigorous justifications of this mechanism in a general setting. In Section \ref{sec:rescaled}, we provide a rigorous justification in a specific rescaled framework, where the notion of alpha-unstable flows naturally arises.
In our setting, the decomposition \eqref{eq:BFBS} approximately takes the form 
$$
B_S = \cS (v) \exp ( i \jj \cdot x + \mu t), \qquad B_L = v \exp (i \jj \cdot x + \mu t), \qquad \jj\in \RR^3\setminus \{0\},
$$
where $v\in \CC^3$ is related to the mean magnetic field. Under suitable assumptions, we obtain a solution to \eqref{passive-vector} up to lower-order terms in $|\jj| \ll 1$ using a perturbative  approach. 
The corresponding  alpha-matrix , up to the large-scale function $\exp (i \jj \cdot x+ \mu t)$, is given by
$$ 
\langle U_S \times B_S \rangle_y  \approx    \fint_{\TT^3} U \times \cS (v) 
$$
and appears in Definition \ref{def:alpha}. Alpha-instability can be explicitly verified for the classical ABC flows, and it is in fact a generic property, see Proposition \ref{prop:alphaDensity}.

 \subsection{Notation}

The symbol $Q_R(x)\subset \RR^3$ denotes the closed ball centered at $x\in \RR^3$ of radius $R>0$. When $x$ is the origin, we simply write $Q_R$. The solution semigroup generated  by \eqref{passive-vector} with vector field $u \in   W^{1, \infty} (\RR^3)$ is denoted  $S_\eps^u(t):L^2(\RR^3)\to L^2(\RR^3)$ and acts as
\begin{equation}\label{d:semigroup}
     B_{\ini}^\eps\mapsto S_\eps^u (t) B_{\ini}^\eps=B^\eps(t)\,.
\end{equation}
Furthermore, we shall denote by $C^\infty_c(\RR^3)$ the space of smooth functions with compact support in $\RR^3$. Given an operator $\cL$, the symbols $\sigma(\cL)$ and $\rho(\cL)$ denote the spectrum and resolvent set, respectively.

\subsection{Strategy of the proof}\label{sub:strategy}

The proof of Theorem \ref{thm:main} combines classical ideas originally introduced by G.O. Roberts in \cites{Roberts70,Roberts72}, a novel rescaling of both the velocity and magnetic fields to derive an $\eps$-independent eigenvalue problem, and the construction of a  velocity field on the whole space by appropriately gluing together rescaled building blocks.

\medskip

\emph{Modal form and the rescaled spectral problem.}
In 1970, G.O. Roberts \cite{Roberts70} proposed to approach the problem by looking at solutions with a particular structure that exhibits exponential growth, similar to the classical normal mode form used in countless linear problems. This constitutes the main idea to construct the \emph{building block} eventually used in our velocity field. We take a periodic, bounded and Lipschitz vector field $U:\TT^3\to\RR^3$,  rescale it as
\begin{equation}\label{eq:Ueps}
    U_\eps (x) = \sqrt\eps  U \left (\frac{x}{\sqrt\eps } \right) \in W^{1, \infty} (\TT^3),\qquad \|U_\eps\|_{W^{1,\infty}}\leq \|U\|_{W^{1,\infty}},
\end{equation}
and look initially for solutions to \eqref{passive-vector} with $u= U_\eps$ of the form
\begin{equation}\label{def:ansatz}
    B^\eps(t,x)= H \left (\frac{x}{\sqrt{\eps}} \right ) \e^{ i \eps^{- \frac12} \jj \cdot x + p t},
\end{equation}
for $p \in \CC$, $\jj\in \RR^3$, and $H:\TT^3\to \RR^3$. Notice that at this stage, $B^\eps$ may not be periodic, nor finite-energy, hence it is not the solution we find in Theorem \ref{thm:main}. However, a computation (c.f. Section \ref{sec:rescaled}) shows that $H$ solves the \emph{$\eps$-independent} eigenvalue problem on $\TT^3$
\begin{align} 
    (\nabla +i  \jj )\times (U \times H) +  (\nabla + i\jj)^2 H &=pH,\label{modal:passive-vector:rescaled} \\
    (\nabla + i \jj ) \cdot H &=0. \label{modal:div}
\end{align}
If $U$ exhibits the \emph{alpha-unstable} (c.f. Definition \ref{def:alpha}), then there exists a triple $(H,\jj,p)$ satisfying the above problem with $\Re (p) >0$.

\begin{proposition} \label{prop:growing}
Suppose that $U \in W^{1, \infty } (\TT^3)$ exhibits the  unstable alpha effect. Then
there exist $\jj_\star \in \RR^3$ with $|\jj_\star| \leq 1$, $p_\star= p(\jj_\star) \in \CC$ with $\Re (p_\star) >0 $ and  $H(\cdot;\jj_\star) \in L^2 (\TT^3)$ solution to \eqref{modal:passive-vector:rescaled}-\eqref{modal:div}.
\end{proposition}

The proof of this result is inspired by the works of Roberts \cites{Roberts70,Roberts72} and relies on classical tools from perturbation theory \cite{K76}, which are briefly introduced in Appendix \ref{appendix}. The perturbation parameter in our analysis is $|\jj|\ll 1$, and the key step involves characterizing the kernel of the linear operator in \eqref{modal:passive-vector:rescaled} at  $\jj=0$. This kernel is three-dimensional, reducing the problem to studying the eigenvalues of a $3\times 3$ matrices. 
Explicit computations in Example \ref{ex:ABCflows} for the so-called ABC flows reveal a broad class of alpha-unstable flows for which Proposition \ref{prop:growing} holds. This insight motivates the construction of a fundamental \emph{building block} that plays a central role in the final vector field $u$ of Theorem \ref{thm:main}.

\medskip

\emph{A Bloch-type theorem and extension to $\RR^3$.}
As mentioned above, formula \eqref{def:ansatz} alone does not directly yield a finite-energy, exponentially growing solution of \eqref{passive-vector}. The issue arises from the $\jj$-dependent plane wave and the periodicity of $H$, which prevent square-integrability. To overcome this, we establish a Bloch-type result (see Lemma \ref{Lemma: Bloch1}), which, when combined with Proposition \ref{prop:growing}, ensures the existence of a finite-energy, exponentially growing solution of \eqref{passive-vector} on the whole space $\RR^3$:
\begin{equation}\label{eq:expgrowth2gamma}
\|B^\eps (t) \|_{L^2 (\RR^3)} \geq \e^{2 \gamma t} \| B_{\ini} \|_{L^2 (\RR^3)}
\end{equation}
where $\gamma >0$ is independent of $\eps$, and the velocity field is given by the periodic extension of the rescaled vector field $U_\eps$ in \eqref{eq:Ueps} - see Theorem \ref{thm:Section2}. This constitutes the main building block in proving Theorem \ref{thm:main}.
However, the rescaled velocity field $U_\eps$ in \eqref{eq:Ueps} still depends on $\eps>0$. To construct a velocity field $u$ that is independent of $\eps$, we glue together multiple translated and compactly supported copies of $U_\eps$ for a carefully chosen sequence of values $\eps \in \{ \zeta^n \}_{n \in \NN}$ with $\zeta \in (0,1)$. The details of this gluing procedure are described below.

A key challenge in this construction is ensuring that the exponential growth rate in \eqref{eq:growthF} remains \emph{uniform} in $\eps$ over each interval $(\zeta^{n+1},\zeta^n]$. Additionally, we must control the concentration of the initial data \eqref{eq:uniform-ini} uniformly in $\eps>0$. These considerations are crucial for maintaining robust control over the dynamics of \eqref{passive-vector} under the ``glued'' velocity field $u$.

\medskip

\emph{The glued vector field and heuristics of the proof.}
We now construct the $\eps$-independent glued vector field and outline the proof of Theorem \ref{thm:main}. The rescaled flow $U_\eps$ in \eqref{eq:Ueps}, derived from an alpha-unstable flow $U$ gives rise - thanks to Proposition \ref{prop:growing} and the Bloch-type result described above - to an exponentially growing solution for any
$\eps>0$, with growth rate $2 \gamma = \Re (p_\star) >0$  independent of $\eps$, see \eqref{eq:expgrowth2gamma}.

To construct an $\eps$-independent divergence-free vector field, we assemble rescaled and compactly supported copies of  $U$ using the sequence $\{ \zeta^n\}_{n \in \NN}$, defining
$$
u= \sum_{n, \ell =1}^\infty u_{n, \ell}\,, \qquad  u_{n, \ell} (x) \approx   \zeta^{n/2} U (\zeta^{-n/2} x) \mathbbm{1}_{Q_{n, \ell}} (x )  \,,
$$
for appropriately chosen disjoint balls $\{ Q_{n, \ell} \}_{n, \ell=1}^\infty$ (see \eqref{eq:AnsatzU} for the precise definition). At this stage, $\ell \in \NN$ is a free parameter. 

For any $\eps \in (0,1)$ and $\zeta \in (0,1)$, there exists a unique $n_\eps \in \NN$ so that $\eps \in (\zeta^{n_\eps +1}, \zeta^{n_\eps} ]$. Accordingly, we define the initial datum as
$$ 
B^\eps_{\ini}=\sum_{\ell=1}^\infty \ell^{-2} F^\eps_\ell 
$$
where each $F^\eps_\ell$ is \emph{concentrated} on $Q_{n_\eps,\ell}$ (see \eqref{eq:condition-initial} for a precise definition).

Now, for any  $t>0$, there exists an index $\ell_t \in \NN$ such that $t \in [\ell_t, \ell_t +1)$. We will show that the solution to \eqref{passive-vector} with initial data $B^\eps_{\ini}$ is \emph{quantitatively} exponentially exponentially large around
$Q_{n_\eps, \ell_t}$ for any  $t \in [\ell_t, \ell_t +1) $ and $\eps \in (\zeta^{n_\eps +1}, \zeta^{n_\eps} ]$, due to the velocity field $u_{n_\eps, \ell_t}$ and the contribution of the initial datum $F_{\ell_t}^{n_\eps}$. More precisely, in Proposition \ref{prop:main} we establish  estimates that imply
\begin{equation}\label{eq:chainineq}
\|S^u_\eps (t) B^\eps_\ini\|_{L^2(\RR^3)}
\geq \|S^u_\eps (t) B^\eps_\ini \|_{L^2(Q_{n_\eps,\ell_t})}  
\gtrsim \ell_t^{-2}\|S^{u_{n_\eps, \ell_t}}_\eps (t)F^\eps_{\ell_t}\|_{L^2(Q_{n_\eps,\ell_t})}
\gtrsim \e^{{\gamma (\ell_t +1)} } \| B_{\ini}^\eps \|_{L^2 (\RR^3)} \,,
\end{equation}
where the constant $\gamma >0$ as well as the implicit constant in $\gtrsim$ are independent of $t>0$ and $ \eps \in (0,1)$.   This chain of inequalities confirms that the solution exhibits exponential growth around the ball $Q_{n_\eps, \ell_t}$ for any  $t \in [\ell_t, \ell_t +1)$ and $\eps \in ( \zeta^{n_\eps +1}, \zeta^{n_\eps}]$. Since  $ t>0$ and $\eps \in (0,1)$ are arbitrary, we conclude the proof. To rigorously justify these estimates, we rely on suitable local energy bounds (see Section \ref{sec:energy-estimates}), together with the careful choice of disjoint balls
$\{ Q_{n, \ell}\}_{n, \ell =1}^\infty$ and appropriate cutoffs.

\section{Spectral analysis and modal form}\label{sec:rescaled}

We now state the main theorem we will prove in this section, which is a key tool used in the proof of Theorem \ref{thm:main} in Section \ref{sec:proof-thm}. 
\begin{MAINtheorem}\label{thm:Section2}
    Let $U = \nabla \times \Psi\in C^\infty(\TT^3)$ be an alpha-unstable velocity.
    Then,  there exist  $\zeta \in (1/2,1)$ and $\gamma>0$ such that 
\begin{equation} 
U_n (x) = \zeta^{n/2} U(\zeta^{-n/2}x), \qquad U_n =\nabla \times \Psi_n,
\end{equation}
satisfies the following properties.
    \begin{itemize}
        \item For every $\eps \in (0,1]$  there exists  $F^\eps\in L^2 (\RR^3)$, normalized to $\|F^\eps\|_{L^2(\RR^3)}=1$, such that if $\eps \in (\zeta^{n_\eps+1}, \zeta^{n_\eps}]$ for some $n_\eps\in\NN$  then
        \begin{equation}\label{eq:growthF}
        \|S^{U_{n_\eps}}_\eps(t) F^\eps\|_{L^2 (\RR^3)}\geq \e^{2\gamma t} \,, \quad \forall t \geq 0 \,.
        \end{equation}
        \item For every $\delta\in (0,1)$, there exists $R>0$ independent of $\eps$ such that 
        \begin{equation} \label{eq:uniform-ini} 
        \|F^\eps\|_{L^2(Q_R)}\geq 1-\delta \,, \quad \forall \eps >0 \,.
        \end{equation}
    \end{itemize}
\end{MAINtheorem}

A consequence of this theorem, considering the case $n=0$ for the velocity field, i.e. $U_0 = U$
and $\eps =1$ for the magnetic diffusivity, is that there exists $F^1 \in L^2 (\RR^3)$ such that
$$ \| S^U_1 (t) F^1  \|_{L^2 (\RR^3)} \geq \e^{2\gamma t} \| F^1 \|_{L^2 (\RR^3)} \,.  $$
In other words, the equation \eqref{passive-vector} admits a solution that grows exponentially in time for 
$\eps = 1$.
The theorem above generalizes this property to a specific rescaling of the velocity field, which preserves a uniform bound on the Lipschitz norm. Furthermore, the initial data corresponding to exponentially growing solutions to \eqref{passive-vector} satisfy a concentration of mass property that holds \emph{uniformly} in 
$\eps$.

The Section is now divided in several subsections.
First, in Subsection \ref{subsec:modal}, we derive an ansatz in the modal form for an exponentially growing solution with a rescaled velocity field. Then, in Subsection \ref{subsec:spectral-analysis}, we analyze the spectral properties of the operator that naturally arises from the eigenvalue problem associated with the exponentially growing solution. Next, in Subsection \ref{subsec:existence-unstable}, we prove Proposition \ref{prop:growing} for the problem on the three-dimensional torus and show that a typical choice of an ABC flow is alpha-unstable. In Subsection \ref{subsec:continuity}, we study the continuity of the operator arising from the eigenvalue problem with respect to magnetic diffusivity. In Subsection \ref{subsec:bloch}, we prove a Bloch-type result, extending our ansatz from the three-dimensional torus to the whole space. Finally, in Subsection \ref{subsec:proof-theorem}, we use these tools to prove Theorem \ref{thm:Section2}.

\subsection{Modal form} \label{subsec:modal}
The goal of this section is to prove Proposition \ref{prop:growing} and construct a solution to \eqref{passive-vector} in $L^2(\RR^3)$. 
To derive the spectral problem \eqref{modal:passive-vector:rescaled}-\eqref{modal:div}, we first set $B^\eps(t,x)= H \left(\frac{x}{\sqrt{\eps}} \right) \e^{ i \eps^{- \frac12} \jj \cdot x + p t}$ as in \eqref{def:ansatz}, use vector calculus identities\footnote{ $\nabla \times (A \times C) = (C \cdot \nabla) A + (\nabla \cdot C) A - (A \cdot \nabla ) C - (\nabla \cdot A) C $ and $A \times (C \times D) =  (A \cdot D) C -  (A \cdot C) D$}  and the divergence free property of $u$ to get 
\begin{align*} 
    \nabla \times (u\times B^\eps) & = \eps^{-\frac12}\e^{i \eps^{-\frac12}\jj \cdot x + p t} \left( \eps^{\frac12}(H \cdot \nabla) u + (\nabla \cdot H) u - (u \cdot \nabla )H + (i\jj\cdot H) u - (i\jj \cdot u) H  \right)\notag \\
    \Delta B^\eps &= \eps^{-1}\e^{pt + \eps^{-\frac12}i\jj \cdot x} ((\nabla + i\jj)^2 H),\\
    \nabla\cdot B^\eps&=(\nabla + i \jj ) \cdot H=0.
\end{align*}
Thus, \eqref{modal:passive-vector:rescaled}-\eqref{modal:div} is derived by  defining $u=U_\eps$ as in \eqref{eq:Ueps}-\eqref{def:ansatz}. 

The main object of our focus is the linear operator $\cL$ in \eqref{modal:passive-vector:rescaled} defined as
\begin{equation}\label{eq:mainop}
    \cL H=(\nabla +i  \jj )\times (U \times H) +  (\nabla + i\jj)^2 H.
\end{equation}
We then define
\begin{equation} 
\label{unperturbed}
\cL_0H=\nabla \times (U \times H)+\Delta H
\end{equation}
and its first order perturbation 
\begin{equation}\label{eq:L1def}
\cL_1H= i\frac{\jj}{|\jj|}\times (U \times H)+2 i\frac{\jj}{|\jj|}\cdot \nabla H.
\end{equation}
In this way,
\begin{equation}\label{def:mainoperatorL}
\cL=\cL_0+| \jj| \cL_1- |\jj|^2,
\end{equation}
and we can appeal to tools from perturbation theory \cite{K76} using $|\jj|\ll 1$ as a small parameter, independently on $\eps$.

\subsection{Spectral analysis} \label{subsec:spectral-analysis}
To analyze the spectral properties of $\cL$ in \eqref{def:mainoperatorL}, we first give a fairly accurate description of $\cL_0$. Its kernel $\Ker(\cL_0)$ was already characterized in the original work of Roberts \cite{Roberts70}, in a perturbative way.

\begin{lemma} \label{lemma:delta-small}
There exists $\delta_0>0$ so that for any divergence-free $U \in W^{1, \infty} (\TT^3)$ with $\|U\|_{W^{1, \infty}}<\delta_0$  and any $v\in \CC^3$, there exists a unique $L^2(\TT^3)$-solution to the problem
\begin{align} \label{eq:SSS}
            \cL_0 S =  \nabla \times ( v \times U ),\qquad 
            \int_{\TT^3} S(x)\dd x =0.
    \end{align}
Calling $\cS:\CC^3\to L^2(\TT^3)$ the corresponding solution map $v\mapsto \cS (v)=S$, we have that
\begin{equation}\label{eq:KerL0}
\Ker(\cL_0)=\left\{v+\cS(v):v\in \CC^3\right\},
\end{equation}
and it is therefore  three-dimensional.  
\end{lemma}

\begin{proof}
We preliminarily notice that $S$ is a mean free solution to \eqref{eq:SSS} if and only if 
\begin{equation}\label{eq:SSSaux}
\left(1+ \nabla \times ( U \times \Delta^{-1})\right)\Delta S  =   \nabla \times ( v \times U ) .
\end{equation}
In particular, note that the invertibility in the space of mean zero functions of the left hand-side of \eqref{eq:SSSaux} is sufficient to guarantee a unique solution to \eqref{eq:SSS}. Furthermore, we have for some $\delta_0\in (0,1)$ that
\begin{align}
\|\nabla \times (U \times \Delta^{-1}S)\|_{L^2} =\|\Delta^{-1}S \cdot \nabla U+\nabla \cdot (\Delta^{-1}S)U-U \cdot \nabla \Delta^{-1}S\|_{L^2} \le  \frac{1}{2\delta_0}\|U\|_{W^{1, \infty}} \|S\|_{L^2},
\end{align}
and therefore, as soon as $\|U\|_{W^{1, \infty}}<\delta_0$, the operator $\Delta+\nabla \times (U \times \cdot)$ is invertible  and 
\begin{equation}\label{eq:NeumannSeries}
(\Delta+\nabla \times (U \times \cdot))^{-1}=\Delta^{-1}\sum_{n \geq 0}(-1)^n (\nabla \times (U \times \Delta^{-1}\cdot))^n.
\end{equation}
To prove \eqref{eq:KerL0}, we first observe that any element of the form $v+\cS(v)$ belongs to the kernel by a straightforward computation. To prove the $\subset$ inclusion, we consider $H\in L^2(\TT^3)$ and write it as 
\begin{equation}
H=\langle H \rangle +H_{\neq}, \qquad \langle H \rangle:= \fint_{\TT^3}H(x)\dd x, \qquad H_{\neq}=H-\langle H \rangle.
\end{equation}
Then $\cL_0 H=0$ implies 
\begin{equation}
\cL_0 H_{\neq} =- \nabla \times (U \times \langle H \rangle)=\nabla \times (\langle H \rangle \times U),
\end{equation}
so that $H_{\neq}=\cS(\langle H \rangle)$. Therefore $H=\langle H \rangle +\cS(\langle H \rangle)$, and the proof is over.
\end{proof}

The Riesz projection of $\cL_0$ onto  $0\in \sigma(\cL_0)$ is defined as
\begin{equation}
    P =\frac{1}{2\pi i} \int_\Gamma (\mu-\cL_0)^{-1}\dd\mu,
\end{equation}
where $\Gamma$ is a contour enclosing a region entirely contained in the resolvent set $\rho(\cL_0)$, and whose interior contains no
spectral points of $\cL_0$ other than zero. This is particularly useful to characterize the spectrum of $\cL$ as follows.

\begin{lemma}
\label{spectral Lemma}
Let $\mu_1, \mu_2, \mu_3$ be the (possibly repeated) eigenvalues $P\cL_1P$ on $PL^2(\TT^3)$. 
Then $P$ maps onto $\Ker(\cL_0)$ and there exists $\delta_1 \in (0,1) $ such that  for all $| \jj | < \delta_1$ there exist three eigenvalues of $\cL$, denoted $p_\ell(\jj)$, $\ell=1,2,3$, which are of the form 
\begin{equation}
\label{eig eq}
p_\ell(\jj)= \mu_\ell| \jj | + o(|\jj|) \,.
\end{equation}
\end{lemma}
\begin{proof}

The proof is an application of Theorem \ref{Kato}. Recalling \eqref{def:mainoperatorL}, it is enough to apply Theorem \ref{Kato} to the operator $\cL_0+ |\jj| \cL_1$, since $\sigma(\cL)=\sigma(\cL_0+ |\jj| \cL_1)-|\jj|^2$ .   Note that $\cL_0$ has a compact resolvent, i.e. $(\mu - \cL_0)^{-1}$ is a compact operator for any $ \mu \in \rho (\cL_0)$, so that in particular, all its eigenvalues are isolated, and the corresponding generalized eigenspaces are finite dimensional, see Lemma \ref{Riesz Projector}. In particular, since $0 \in \sigma(\cL_0)$, we only need to check that the eigenvalue $0$ is  semisimple (which also implies that $P$ maps onto $\Ker(\cL_0)$), and that the operator $\cL_1$ satisfies $\|\cL_1H\|\lesssim  \|\cL_0H\|+ \|H\|$.

\medskip

\emph{Step 1: $0$ is semisimple.} Suppose that $\cL_0^2 v=0$. In view of Lemma \ref{lemma:delta-small}, $\cL_0v=w+\cS(w)$, where $w\in \CC^3$. Since $\cL_0v$ and $\cS(w)$ have zero average, this implies that $w=0$. In turn, $\cS(w)=0$ and $\cL_0v=0$ as well. 

\medskip

\emph{Step 2: Estimate on $\cL_1$.}
By a straightforward computation,
\begin{equation}
-2\langle \nabla \times  (U  \times H), H \rangle   + 2\|\nabla H\|_{L^2}^2=-2\langle \cL_0 H,H\rangle
\leq  \|\cL_0 H\|^2_{L^2}+  \| H\|^2_{L^2}
\end{equation}
An integration by parts shows that
\begin{equation}
2|\langle \nabla \times  (U  \times H), H \rangle|
\lesssim\|U\|_{L^\infty} \|H\|_{L^2}\|\nabla H\|_{L^2},
\end{equation}
so that 
\begin{equation}\label{eq:grad-est}
\|\nabla H\|_{L^2}^2\lesssim  \|\cL_0 H\|^2_{L^2}+  (1+\|U\|_{L^\infty}^2)\| H\|^2_{L^2}
\end{equation}
It is clear from \eqref{eq:L1def} that
\begin{equation}
\|\cL_1H\|^2 \lesssim \|\nabla H\|^2+\|U\|_{L^\infty}^2\|H\|^2,
\end{equation}
and thus the conclusion of the proof follows from combining the above estimate with \eqref{eq:grad-est}.
\end{proof}

Thanks to Lemmas \ref{lemma:delta-small}-\ref{spectral Lemma}, the eigenvalue problem is reduced to a finite dimensional one. The following property is crucial to understand it in detail.

\begin{lemma}
\label{averaging lemma}
The Riesz projector $P$ and averaging $\langle\cdot\rangle$ satisfy
\begin{equation}
\langle PH \rangle =\langle H \rangle, \qquad PH=\langle H \rangle+ \cS(\langle H \rangle),
\end{equation}
for all $H \in L^2(\TT^3)$.
\end{lemma}

\begin{proof}
Fix $\mu \in \CC\setminus\{0\}$. By definition of the resolvent $R(\mu;\cL_0) = (\mu - \cL_0 )^{-1}$, we have  $(\mu - \cL_0) R(\mu;\cL_0) H =H$, which we can write explicitly as
\begin{equation}
\mu R(\mu;\cL_0) H-\Delta (R(\mu;\cL_0)H)-\nabla \times (u\times (R(\mu;\cL_0)H))=H.
\end{equation}
Applying $\langle \cdot \rangle$ to this equation, we obtain
\begin{equation}
\mu \langle R(\mu;\cL_0) H \rangle=\langle H\rangle, \qquad \mu\in \CC\setminus\{0\},
\end{equation}
so that $\langle R(\mu;\cL_0) H\rangle=\mu^{-1} \langle H \rangle$ for any $H\in L^2(\TT^3)$. Hence, it remains to apply $\langle \cdot \rangle$ to the definition of $P$ which yields 
\begin{equation}
\langle P H\rangle =\frac{1}{2\pi i}\int_\Gamma \langle R(\mu;\cL_0)H \rangle\dd \mu =\frac{1}{2\pi i}\int_{\Gamma} \mu^{-1}\langle H \rangle \dd\mu=\langle H \rangle,
\end{equation}
as required. 
Thanks to Lemma \ref{spectral Lemma}, $P$ maps onto  $\Ker(\cL_0)$ which is precisely those vector fields of the form $v+\cS(v)$, for some  $v \in \CC^3$, thanks to \eqref{eq:KerL0}. Since $PH \in \Ker(\cL_0)$ and $\langle PH \rangle=\langle H \rangle$, the expression for $P$ follows immediately.
\end{proof}

It remains to characterize the eigenvalues of $P\cL_1P$ in a way that gives us the possibility to control their behavior and, in particular, the positivity of their real part.

\begin{lemma} \label{lemma:relation-eigenvalues}
Under the assumptions of Lemma \ref{spectral Lemma}, $\mu_\ell$ is an eigenvalue of $P\cL_1P$ on $PL^2(\TT^3)$ if and only if there exists some $v_\ell \in \CC^3$ satisfying
\begin{equation}
i \frac{\jj}{|\jj|} \times\fint_{\TT^3}   U \times \cS(v_\ell)  =\mu_\ell v_\ell.
\end{equation}
\end{lemma}

\begin{proof}
Fix $\jj\in \RR^3$ with $|\jj|<\delta_1$ as in Lemma \ref{spectral Lemma}, and suppose that $P\cL_1 Pv=\lambda v$ for some $v \in PL^2(\TT^3)$. We observe that $P$ maps onto the kernel of $\cL_0$ and by Lemma  \ref{lemma:delta-small} we have that $v\in \Ker(\cL_0)$ can be decomposed as $v=w+\cS(w)$ for some $w \in \CC^3$ with $\langle \cS (w) \rangle =0$.
Then, by Lemma \ref{averaging lemma}, we may apply $\langle\cdot\rangle$ to both sides of the equation to see 
\begin{equation}
\label{averaged step 1}
\langle \cL_1Pv \rangle=\lambda  w \,.
\end{equation}
 Substituting this into the expression \eqref{averaged step 1} and noting that the $\frac{\jj}{|\jj|} \cdot \nabla $ contribution of $\cL_1$ maps onto  zero average functions,  we see that
\begin{equation}
i \frac{\jj}{|\jj|}\times\fint_{\TT^3} U \times (w+\cS(w)) =\lambda w.
\end{equation}
Finally, since $U$ is zero average, we may further reduce this to 
\begin{equation}
\label{averaged eig equation}
i \frac{\jj}{|\jj|}\times \fint_{\TT^3}  U \times \cS(w) =\lambda w,
\end{equation}
showing the first direction of the implication. To prove the reverse direction, suppose that 
\begin{equation}
i \frac{\jj}{|\jj|}\times \fint_{\TT^3}  U \times \cS(v_\ell) =\mu_\ell v_\ell,\qquad \ell=1,2,3.
\end{equation}
Set $H=v_\ell +\cS(v_\ell)$. Applying $\cL_1$ and decomposing the result into a mean-free and constant parts yields 
\begin{equation}
\cL_1H= i\frac{\jj}{|\jj|} \times \langle U \times  \cS(v_\ell)\rangle+i\frac{\jj}{|\jj|} \times (U \times  \cS(v_\ell))_{\neq}
+i\frac{\jj}{|\jj|} \times (U \times v_\ell)+2 \frac{\jj}{|\jj|} \cdot \nabla \cS(v_\ell).
\end{equation}
Thus, applying $P$ and using its explicit form we found in Lemma \ref{averaging lemma}, we deduce that 
\begin{equation}
P\cL_1H
=\langle \cL_1H \rangle +\cS(\langle\cL_1H\rangle) 
= i\frac{\jj}{|\jj|} \times \langle  U \times  \cS(v_\ell)\rangle +\cS\left( i\frac{\jj}{|\jj|} \times \langle U \times  \cS(v_\ell)\rangle\right).
\end{equation}
Now, $ i\frac{\jj}{|\jj|} \times \langle U \times \cS(v_\ell)\rangle=\mu_\ell v_\ell$ by assumption. Therefore, we get
\begin{equation}
P\cL_1H=\mu_\ell (v_\ell+\cS(v_\ell))=\mu_\ell H,
\end{equation}
which completes the proof.
\end{proof}

\subsection{Existence of an unstable eigenvalue and  alpha-unstable flows} \label{subsec:existence-unstable}

The characterization of eigenvalues given in Lemma \ref{lemma:relation-eigenvalues} is precisely the motivation behind the instability condition \eqref{eq:Amatrix} in Definition \ref{def:alpha}.
With this definition at hand, the proof of Proposition \ref{prop:growing} readily follows. 

\begin{proof}[Proof of Proposition \ref{prop:growing}]
By assumption we know that there exists $\jj, v\in \CC^3$ and $\mu \in \CC$ with $\Re(\mu) >0$ such that \eqref{eq:Amatrix} holds. Using Lemma \ref{spectral Lemma} and Lemma \ref{lemma:relation-eigenvalues} we have that $\mu$ is an eigenvalue of $P\cL_1P$, and therefore $\cL$ has an eigenvalue of the form
\begin{equation} 
p(\jj) =  \mu|\jj|  + o (|\jj|).
\end{equation}
Up to replacing $\jj$ by $\beta \jj$ for $\beta \in (0,1)$ small enough, we get that $p(\jj) $ is such that $\Re (p(\jj)) >0$. Notice  that the equation for $\cS (v)$ \eqref{eq:SSS} is independent of $\jj$. 

To conclude, we need to verify the modal divergence condition \eqref{modal:div}. This is in fact equivalent to verifying the usual divergence-free condition for $B(x)=\e^{i\jj\cdot x}H(x)$, which solves 
\begin{equation}
\nabla \times (U \times B)+\Delta B=pB
\end{equation}
in a distributional sense. In particular, we may take divergence of the equation, to see that, distributionally it holds 
\begin{equation}
\Delta \nabla \cdot B=p \nabla \cdot B.
\end{equation}
on $\RR^3$. Since $\Re(p)>0$, and $\e^{i\jj\cdot x}H(x)$ is a tempered distribution, we may take the distributional Fourier transform and conclude $\nabla \cdot B=0$.
This concludes the proof.
\end{proof}

Finally, we provide an observation suggesting that the property of having a positive eigenvalue is in some sense generic for the profile $U$.
\begin{lemma}
Let $\mu_\ell$ be as in Lemma \ref{spectral Lemma}. Then one and only one of the following is true: 
\begin{enumerate}
    \item For all $\jj \in \RR^3$ and all $\ell=1,2,3$, it holds $\Re(\mu_\ell)=0$.
    \item There exists a nonzero $\jj\in\RR^3$ and $\ell=1,2,3$, so that $\Re(\mu_\ell)>0$.
\end{enumerate}
\end{lemma}
\begin{proof}
This is a straightforward result. Indeed, suppose that there exists some $\mu_\ell$ with nonzero real part, i.e. there exists a nonzero $\jj$ in $\RR^3$, and a vector $v \in \CC^3$ so that 
\begin{equation}
\fint_{\TT^3} i \frac{\jj}{|\jj|}\times (U \times \cS(v))=\mu_\ell v,
\end{equation}
If $\Re(\mu_\ell)>0$ we are done. Otherwise, simply replace $\jj \mapsto -\jj$, which completes the proof.
\end{proof}

As we shall see in Proposition \ref{prop:alphaDensity}, alpha-instability is a \emph{generic} property of smooth velocity fields. Before dealing with this, we show that the classical ABC flow is an explicit example   of velocity field that is alpha-unstable.

\begin{example}[ABC flows]\label{ex:ABCflows}
The ABC flows (named after Arnold, Beltrami, and Childress) are a family of steady, three-dimensional velocity fields that solve the $3d$ Euler equations. Given amplitute parameters $\sfa,\sfb,\sfc\in \RR$, the corresponding velocity field is given by
\begin{equation} \label{eq:ABC}
\Uabc(x_1,x_2,x_3)=\begin{pmatrix}
\sfa \sin(x_3)+\sfc \cos(x_2)\\
\sfb \sin(x_1)+ \sfa \cos(x_3)\\
\sfc \sin(x_2) +\sfb \cos(x_1)
\end{pmatrix}.
\end{equation}
Since
\begin{equation}
\|\Uabc\|_{W^{1,\infty}(\TT^3)}=\max\{|\sfa|+|\sfc|, |\sfa|+|\sfb|,|\sfb|+|\sfc|\},
\end{equation}
we consider a small positive parameter  $\delta_0 \ll 1$ to be able to apply Lemma \ref{lemma:delta-small} for $U=\delta_0 \Uabc$. In particular, the solution operator of \eqref{eq:SSS} can be written as in \eqref{eq:NeumannSeries} as
\begin{equation} \label{eq:expansion}
\cS(v)=\Delta^{-1}\sum_{n \geq 0}(-1)^n  (\nabla \times (\delta_0\Uabc \times \Delta^{-1} ))^n G_{\delta_0}(v),
\end{equation}
where $G_{\delta_0}(v)=- \nabla \times (\delta_0\Uabc \times v)$. Inserting this into the expression \eqref{eq:Amatrix} for the alpha-instability, we see that we need to compute
\begin{equation} \label{eq:full-matrix}
\fint_{\TT^3} i \frac{\jj}{|\jj|}\times (\delta_0 \Uabc \times \cS(v))=
\sum_{n \geq 0} (-1)^n i \frac{\jj}{|\jj|}\times \fint_{\TT^3} \delta_0 \Uabc \times   \Delta^{-1}\left(\nabla \times (\delta_0 \Uabc \times \Delta^{-1} )\right)^n G_{\delta_0}(v).
\end{equation}
In particular, at lowest order in $\delta_0$, i.e. the $n=0$ term above, this is 
\begin{equation}
\label{eq: first order matrix}
i\delta_0^{2}  \frac{\jj}{|\jj|} \times \cI(v),
 \qquad \cI(v):= -\fint_{\TT^3} \Uabc \times \Delta^{-1} (\nabla \times (\Uabc \times v)).
\end{equation}
For sufficiently small 
$\delta_0 $ we can apply the perturbative arguments from Lemma \ref{Abstract spectral Lemma} to show that if \eqref{eq: first order matrix} has a simple eigenvalue with a positive real part, then the full matrix \eqref{eq:full-matrix} will as well.

Using the fact that $\Uabc$ is divergence-free and $v$ is constant, we have
\begin{equation}\label{eq:IvFourier1}
\cI(v)= -\fint_{\TT^3} \Uabc \times \Delta^{-1} ((v\cdot i\nabla)  \Uabc  )=\sum_{\bk \in \ZZ^3, |\bk|=1} \hUabc(\bk)\times\left((v\cdot \bk)  \hUabc(-\bk)  \right),
\end{equation}
where we used that $\hUabc$ is concentrated in the modes with $|\bk|=1$. Indeed, we can compute explicitly 
\begin{equation}
\hUabc(\bk)=\frac{\sfa}{2} \begin{pmatrix} \pm i\\ 1 \\0\end{pmatrix}\mathds{1}_{(0,0, \pm 1)}(\bk)
+\frac{\sfb}{2}\begin{pmatrix} 0 \\ \pm i \\ 1\end{pmatrix}\mathds{1}_{(\pm 1,0,0)}(\bk)
+\frac{\sfc}{2} \begin{pmatrix} 1\\0\\ \pm i\end{pmatrix}
\mathds{1}_{(0,\pm 1,0)}(\bk).
\end{equation}
From this, it follows that
\begin{equation}\label{eq:IvFourier2}
\hUabc(\bk)\times  \hUabc(-\bk)= 
\pm i\frac{\sfa^2}{2} \begin{pmatrix} 0\\ 0 \\1\end{pmatrix}\mathds{1}_{(0,0, \pm 1)}(\bk)
\pm i\frac{\sfb^2}{2}\begin{pmatrix}  1 \\0 \\ 0\end{pmatrix}\mathds{1}_{(\pm 1,0,0)}(\bk)
\pm i \frac{\sfc^2}{2} \begin{pmatrix} 0\\ 1\\ 0\end{pmatrix}
\mathds{1}_{(0,\pm 1,0)}(\bk),
\end{equation}
and thus a computation shows
\begin{equation}
    \cI(v)= J v, \qquad
    J:=\begin{pmatrix} \sfb^2 &0&0\\ 0& \sfc^2 &0 \\ 0 & 0 &\sfa^2 \end{pmatrix},\qquad v\in\CC^3.
\end{equation}
In turn,
\begin{equation}
i\frac{\jj}{|\jj|} \times \cI(v)= L v,\qquad 
L:=\frac{i}{|\jj|}
\begin{pmatrix} 
0&- \sfc^2j_3&  \sfa^2 j_2\\ 
\sfb^2 j_3 & 0&-\sfa^2j_1  \\ 
-\sfb^2 j_2 & \sfc^2j_1 & 0 
\end{pmatrix},
\qquad \jj=(j_1,j_2,j_3)\in \RR^3.
\end{equation}
In this case, the alpha-instability condition \eqref{eq:Amatrix} is equivalent to finding a nonzero  $\jj\in \RR^3$ such that (recall \eqref{eq: first order matrix}) the matrix $L$ above has an eigenvalue with positive real part. The three eigenvalues of $L$ can be explicitly computed as
\begin{equation}
    \mu_0=0, \qquad \mu_{\pm} = \pm \frac{1}{|\jj|}\sqrt{\sfa^2 \sfb^2 j_2^2 + \sfb^2 \sfc^2 j_3^2 + \sfa^2\sfc^2  j_1^2}\,.
\end{equation}
In particular, as long as any two of $\sfa,\sfb,\sfc$ are nonzero, there exists a nonzero $\jj \in \RR^3$  so that the ABC flow is alpha-unstable.
\end{example}

 Let $\dot{H}^\infty_{\text{div}}$ the space of divergence-free $H^\infty$ vector fields with zero average. The topology on $\dot{H}^\infty_{\text{div}}$ is induced by the countable family of seminorms $\|\cdot \|_{H^n}$, for $n \in \NN$. Consequently, $\dot{H}^\infty_{\text{div}}$ is  metrizable, with a metric given, for instance, by 
\begin{equation}
\dd_{H^\infty}(U,V)=\sum_{n \geq 0}2^{-n}\frac{\|U-V\|_{H^n}}{1+\|U-V\|_{H^n}} \, .
\end{equation}
With this example in hand, we can establish that this property holds for a dense set of velocity profiles  $U$. More precisely, for any $U \in \dot{H}^\infty_{\text{div}}$, we can use the matrix $A$ defined in \eqref{eq:Amatrix} and define 
\begin{equation}
 \Omega:= \{ U \in \dot{H}^\infty_{\text{div}} : \exists \delta_0>0, \jj\in\RR^3 \text{ s.t. } A (\delta_0 U, \jj) \text{ has a simple eigenvalue $p$ with } \Re (p) >0  \}.
\end{equation}
\begin{proposition}\label{prop:alphaDensity}
    The set $\Omega$ is open and dense in $\dot{H}^\infty_{\text{div}}$. 
\end{proposition}
\begin{remark}
    The small parameter $\delta_0>0$ primarily serves to simplify computations by enabling the justification of the series expansion \eqref{eq:expansion}. This, in turn, provides a first-order approximation of the operator $S(v)$ making its analysis more tractable. 
\end{remark}
\begin{proof}[Proof of Proposition \ref{prop:alphaDensity}]
This proof proceeds along much the same lines as that of Lemma 3.2. in \cite{Gerardvaret05}. To establish the density claim, we begin by noting that, as in \eqref{eq:NeumannSeries}, for any $U \in \dot{H}^\infty_{\text{div}}$ and sufficiently small $\delta_0 \ll 1$, we have 
\begin{equation}
\fint_{\TT^3} (\delta_0 U\times S(v))  = \sum_{n \geq 0}(-1)^n\fint_{\TT^3} \delta_0 U \times \Delta^{-1}\left(\nabla \times \left(\delta_0 U\times \Delta^{-1}\right)\right)^nG_{\delta_0} (v) 
\end{equation}
with $G_{\delta_0} (v)=-\nabla \times (\delta_0 U \times v)$. In particular, we may write  
\begin{equation}
\fint_{\TT^3} (\delta_0 U\times S(v))=\delta_0^2\left[-\fint_{\TT^3} U \times \Delta^{-1}\nabla \times (U \times v) +B_{\delta_0}(v)\right]=\delta_0^2 \left[ \cI_{U}(v)+B_{\delta_0}(v)\right]
\end{equation}
where $\|B_{\delta_0}\|=O(\delta_0)$, and 
\begin{equation} \label{eq:IU}
\cI_{U}(v):= -\fint_{\TT^3} U \times \Delta^{-1} (\nabla \times (U \times v)).
\end{equation} 
Taking the cross product with $\jj\in \mathbb{S}^2$,  Lemma \ref{Abstract spectral Lemma} implies that if $i \jj \times \cI_{U}(v)$
has a simple eigenvalue with positive real part, then for sufficiently small $\delta_0$ the same holds for $A(\delta_0 U, \jj)$. 
Thus, to establish density, it suffices to show that the set
\begin{equation}
\Omega_0:=\{U \in \dot{H}^\infty_{\text{div}}: i \jj \times \cI_U   \text{ has a simple eigenvalue with positive real part for some $\jj \in \mathbb{S}^2$}\}
\end{equation}
is dense in $\dot{H}^\infty_{\text{div}}$. 

However, we can further reduce the complexity of this problem. Specifically, for any real velocity field $U$, the matrix $\cI_{U}$ is self-adjoint. To see this, note that since $\cI_{U}$ is real-valued, it suffices to show that it is symmetric. Given $v,w \in \RR^3$, we use the anti-self-adjoint property of the cross product and the self-adjointness of the curl operator in $L^2$ to obtain 
\begin{equation}
\langle v, \cI_U(w)\rangle=\fint_{\TT^3}\langle U \times v, \nabla \times (\Delta^{-1} U \times w)\rangle =-\fint_{\TT^3}\langle  U \times (\Delta^{-1}\nabla \times (U \times v)),w \rangle =\langle \cI_{U}(v),w\rangle.
\end{equation}
Thus, $\cI_U$ is symmetric, allowing us to write it in the form
 $\cI_U=O^\top D O$, where $O$ is an orthogonal matrix and $D$ is a diagonal matrix given by $D= \text{diag} (\alpha_1, \alpha_2, \alpha_3)$, with its diagonal entries representing the real eigenvalues of $\cI_U$. 
 
 We now analyze the matrix $i\jj \times \cI_U=i\jj\times O^\top DO$. 
Using the identity 
$$
Mv\times Mw=\text{det} (M) M^{-T} (v \times w)
$$
for any matrix $M$ and  viewing  $O \jj \times D$ as a linear map\footnote{Denoting by $D_n$ the n$^{\text{th}}$ column of the matrix $D$, the matrix $O\jj \times D$ can be explicitly computed to be  $( O \jj \times D_1, O \jj \times D_2, O \jj \times D_3).$}, we obtain
\begin{equation}
i\jj\times (O^\top  DOv)=O^\top  (i O\jj \times DOv)=O^\top  (i O\jj \times D) Ov \,,
\end{equation}
Thus, it suffices to study the eigenvalues of $i(O\jj)\times D$. Finally, since $O$ is an isometry of $\RR^3$, it maps $\mathbb{S}^2$ onto itself, allowing us to replace $O\jj$ with any $\jj \in \mathbb{S}^2$. This flexibility means we retain full freedom in choosing
 $\jj$ in the Definition \ref{def:alpha} of alpha-instability. 
Hence, it remains to study the spectrum of 
\begin{equation}
i \jj \times D =i\begin{pmatrix}
0 & -\alpha_2 j_3 & \alpha_3 j_2\\
\alpha_1 j_3 & 0 & -\alpha_3 j_1\\
-\alpha_1 j_2 & \alpha_2 j_1 & 0
\end{pmatrix} \,.
\end{equation}
 A simple computation shows that the eigenvalues are given by 
\begin{equation}
\mu_0=0, \qquad \mu_{\pm}=\pm \sqrt{\alpha_1 \alpha_3 j_2^2+\alpha_1 \alpha_2 j_3^2+\alpha_2 \alpha_3 j_1^2}.
\end{equation}
As long as all of $\alpha_1,\alpha_2, \alpha_3$ are non-zero, at least one product - without loss of generality we can take $\alpha_1 \alpha_3$ - must be strictly positive. Setting $\jj=(1,0,0)$, we obtain a simple eigenvalue with a positive real part. 

Thus, it suffices to prove that the set 
\begin{equation}
\Omega_1:=\{U \in \dot{H}^\infty_{\text{div}}: \Ker(\cI_U )=\{0\}\}
\end{equation}
is dense in $\dot{H}^\infty_{\text{div}}$.
Fix $\epsilon>0$, and let $U_N =P_{|\bk | \leq N} U$, where $P_{|\bk| \leq N}$ is the Fourier projection onto $|\bk| \leq N$ modes, and $N=N(\epsilon)$ is big enough such that  
\begin{equation}
\dd_{H^\infty}(U_{N},U) < \frac{\epsilon}{2} \,. 
\end{equation}
    If $\cI_{U_{N}}$has a trivial kernel, the proof is complete.  
    Otherwise, for $\tilde\epsilon\in (0,1)$, we consider 
\begin{equation}
\widetilde{U}_{N} = U_{N} + \tilde{\epsilon} \,U_{1}^{N+1}
\end{equation}
where $U_{1}$ is the ABC flow in \eqref{eq:ABC} with $\sfa=\sfb=\sfc=1$
and 
\begin{equation}
U_{1}^{M}(\cdot) = U_{1} (M \cdot), \qquad \forall M \in \NN .
\end{equation}
For sufficiently small $\tilde{\epsilon}$, it follows that 
\begin{equation} 
\dd_{H^\infty}(\widetilde{U}_{N},U_{N}) < \frac{\epsilon}{2} ,
\end{equation}
so it remains to show that $\widetilde{U}_{N}$ has a simple eigenvalue with a positive real part. To do this, we claim from \eqref{eq:IU}  that
\begin{equation}
\cI_{\widetilde{U}_{N}}=\cI_{U_{N}}+\tilde{\epsilon}^2\cI_{U^{N+1}_1}.
\end{equation}
Indeed, we have the following expression:
\begin{align}\label{eq:compFourierU}
\cI_{\widetilde{U}_{N}}(v)&=\cI_{U_{N}}(v)+\tilde{\epsilon}^2 \cI_{U_{1}^{N+1}}(v)\notag\\
&\quad -\tilde{\epsilon}\left[\fint_{\TT^3}U_{N} \times \Delta^{-1}\nabla \times (U^{N+1}_{1}\times v)+\fint_{\TT^3}U^{N+1}_{1} \times \Delta^{-1}\nabla \times (U_{N} \times v)\right].
\end{align}
In Fourier space, we compute  
\begin{equation}
\fint_{\TT^3}U_{N} \times \Delta^{-1}\nabla \times (U^{N+1}_{1}\times v)=  - \sum_{\bk \in \ZZ^3}\widehat{U}_{N}(\bk)\times i\bk \times \left(|\bk|^{-2} \widehat{U}^{N+1}_{1}(\bk)\times v\right).
\end{equation}
Now, observe that $U_{N}$ is supported on Fourier modes with $|\bk|\leq N$, whereas $U^{N+1}_{1}$ is supported solely on Fourier modes with $|\bk|=N+1$. Thus, the sum vanishes. The second term in \eqref{eq:compFourierU} is handled in an identical manner. Finally, since $\cI_{U^{N+1}_{1}}=\Id$, it follows that we can take $\tilde{\epsilon}$ sufficiently small so that $\cI_{\widetilde{U}_{N}}$ has trivial kernel, thus completing the proof of density.

Next, we prove the openness. By a direct computation similar to \eqref{eq:compFourierU}, we know that if  $U \in \Omega$ and $\widetilde{U} \in H^\infty$, then 
\begin{equation}
\cI_{U+\widetilde{U}}(v)-\cI_{U}(v)=\cI_{\widetilde{U}}(v)-\fint_{\TT^3} U \times \Delta^{-1} \nabla \times (\widetilde{U}\times v)-\fint_{\TT^3}\widetilde{U}\times \Delta^{-1}\nabla \times ( U \times v).
\end{equation}
By applying Cauchy-Schwarz, we obtain
\begin{equation}
\left|\cI_{U+\widetilde{U}}(v)-\cI_{U}(v)\right| \leq \|U\|_{L^2}\|\widetilde{U}\|_{L^2}|v|,
\end{equation}
and similarly for the other terms.
Therefore, as long as $\dd_{H^\infty}(U,U+\widetilde{U})<\epsilon$ we have 
\begin{equation}
 \|\cI_{U+\widetilde{U}}-\cI_{U}\| \lesssim (1+\|U\|_{L^2})\epsilon.
\end{equation}
Thus, by Lemma \ref{Abstract spectral Lemma}, for sufficiently small $\epsilon$   depending only on $\|U\|_{L^2}$  and the spectrum of $ \cI_{U}$,   the matrix $\cI_{U+\widetilde{U}}$ still has an empty kernel, so that $i\jj \times A(\delta_0 U, \jj)$ has a simple eigenvalue with positive real part for some $\jj \in \mathbb{S}^2, \delta_0>0$, completing the proof of openness.
\end{proof}

\subsection{Continuity with respect to diffusivity} \label{subsec:continuity}
We now undertake a more detailed analysis of the spectral properties of the passive vector operator. Thus far, we have shown that alpha-instability guarantees the existence of an unstable eigenvalue for the operator $\cL$ when $|\jj|$ is sufficiently small and the diffusivity parameter is fixed. However, it is essential to establish that this unstable eigenvalue persists under small perturbations of both the mode $\jj$ and the diffusivity parameter $\eps$.  

Since these perturbations are relatively bounded with respect to $\cL$, we can invoke the classical perturbation theory of Kato \cite{K76}, as outlined in Appendix \ref{appendix}. With this in mind, we first derive an elementary resolvent bound estimate for $\cL$. Given that our primary focus is on the continuity of $\cL$ with respect to $\jj$ and $\eps$, we will make this dependence explicit in the notation and write
\begin{equation}
\cL(\jj,\eps)H=(\nabla +i \jj)\times (U \times H)+\eps(\nabla +i \jj)^2 H.
\end{equation}
 We now prove that $\cL (\jj, \eps)$ is Lipschitz continuous in $\jj, \eps$ upon composing on the right with its resolvent. 
\begin{lemma}
\label{resolvent bounds}
Let $U \in C^\infty (\TT^3)$ be a divergence-free velocity field with zero average, and let  $\eps >0$, $\jj \in \mathbb{R}^3$. For any $\mu \in \rho (\cL(\jj, \eps))$, the following bound holds:
\begin{equation}\label{eq:resolventDifference}
\|(\cL( \jj,\eps)-\cL(\jj ',\eps '))(\cL(\jj,\eps)-\mu)^{-1}\|_{L^2\to L^2}\leq C  \left( |\jj-\jj '|+ | \eps - \eps'|\right)(\|(\cL(\jj,\eps)-\mu)^{-1}\|_{L^2\to L^2}+1)
\end{equation}
for any $\jj, \jj', \eps, \eps'$,
where the constant $C>0$ depends continuously on $\jj, \jj', \eps,\eps',\mu$.  In particular, $(\cL(\jj,\eps)-\mu)^{-1}$ is jointly continuous in its arguments, so that for any $\jj,\eps, \mu$ with $\mu \in \rho(\cL(\jj,\eps))$, there holds
\begin{equation}
\label{eq:uniform Resolvent bound}
\|(\cL( \jj,\eps)-\cL(\jj ',\eps '))(\cL(\jj,\eps)-\mu)^{-1}\|_{L^2\to L^2}\leq C  \left( |\jj-\jj '|+ | \eps - \eps'|\right)
\end{equation}
where now the constant $C>$ depends continuously on $\jj, \jj', \eps,\eps', \mu$ with $\mu \in \rho(\cL(\jj,\eps))$.
\end{lemma}  

\begin{remark}
    The constant $C>0 $ in Lemma \ref{resolvent bounds} may diverge as $\eps,\eps'\to 0$.  However, in our analysis, we will only apply the lemma for the case $\eps =1$  and $\eps'\to1$.
\end{remark} 
\begin{proof}[Proof of Lemma \ref{resolvent bounds}]
A computation shows that $ \left( \cL(\jj,\eps)-\cL(\jj',\eps') \right)H$ is equal to
\begin{equation*}
i(\jj-\jj' )\times (U \times H) +\eps(\Delta H+2i \jj \cdot \nabla  H-|\jj|^2H)-\eps'(\Delta H+2i \jj' \cdot \nabla  H-|\jj'|^2H)
\end{equation*}
and so, bounding this expression  we get 
\begin{equation*}
\|\left(\cL(\jj,\eps)-\cL(\jj',\eps') \right) H\|_{L^2}\lesssim \big[ |\jj -\jj'| \left( 1+|\jj'|+|\jj | \right)+  |\eps-\eps'| (1+|\jj|) +\eps'|\jj ' -\jj| )\big](\| H\|_{L^2}+\|\Delta H\|_{L^2}).
\end{equation*}
If we set $h=(\cL(\jj,\eps)-\mu)H$, estimate \eqref{eq:resolventDifference} is proven once we prove a bound on $\Delta H$.
Rearranging the equation for $\cL(\jj,\eps)$, we see that for any $H$ there holds 
\begin{equation}
\eps \Delta H =(\cL(\eps,\jj)-\mu)H-\eps (2i \jj \cdot \nabla -|\jj|^2)H+\mu H -  (\nabla +i \jj)\times (U \times H).
\end{equation}
Therefore, using the definition of the resolvent and an interpolation inequality, there holds 
\begin{align*}
\|\Delta (\cL(\jj,\eps)-\mu)^{-1}h\|_{L^2} 
&\leq C(\eps,\jj)\left(\|h\|_{L^2}+\|(\cL(\jj,\eps)-\mu)^{-1}h\|_{L^2}+\|\nabla (\cL(\jj,\eps)-\mu)^{-1}h\|_{L^2}\right) \\
&\leq C(\eps,\jj)\left(\|h\|_{L^2}+\|(\cL(\jj,\eps)-\mu)^{-1}h\|_{L^2}\right)+\frac12\|\Delta (\cL(\jj,\eps)-\mu)^{-1}h\|_{L^2}.
\end{align*}
Rearranging, this proves the first claim of the theorem.

It thus remains to prove \eqref{eq:uniform Resolvent bound}. To do so, it suffices to prove that the term $\|(\cL(\jj,\eps)-\mu)^{-1}\|_{L^2\to L^2}$ is jointly continuous in $\jj,\eps,\mu$. To do so, note that 
\begin{align*}
\cL(\jj',\eps')-\mu'& =\cL(\jj,\eps)-\mu+[\cL(\jj',\eps')-\cL(\jj,\eps)+(\mu-\mu')]
\\
& =\left\{1+[\cL(\jj',\eps')-\cL(\jj,\eps)+(\mu-\mu')](\cL(\jj,\eps)-\mu)^{-1}\right\}(\cL(\jj,\eps)-\mu).
\end{align*}
We now recall the bound \eqref{eq:resolventDifference} and defining $$
M=\|[\cL(\jj',\eps')-\cL(\jj,\eps)+(\mu-\mu')](\cL(\jj,\eps)-\mu)^{-1})\|_{L^2\to L^2},
$$
it holds that
\begin{equation}
M \leq C(|\jj-\jj'|+|\eps-\eps'|+|\mu-\mu'|)(\|\cL(\jj,\eps)-\mu)^{-1}\|_{L^2\to L^2}+1).
\end{equation}
In particular, as $(\jj',\eps',\mu') \to (\jj,\eps,\mu)$, it follows that $M \to 0$.
Therefore, as long as $M<1$, we can use a Neumann series expansion to write 
\begin{align*}
&(\cL(\jj',\eps')-\mu')^{-1}-(\cL(\jj,\eps)-\mu)^{-1}
=\\
&\qquad\qquad (\cL(\jj,\eps)-\mu)^{-1}\sum_{n \geq 1}(-1)^n \left\{ [\cL(\jj',\eps')-\cL(\jj,\eps)+(\mu-\mu')](\cL(\jj,\eps)-\mu)^{-1}\right\}^n.
\end{align*}
Taking norms of this equation, we see that 
\begin{equation}
\|(\cL(\jj',\eps')-\mu')^{-1}-(\cL(\jj,\eps)-\mu)^{-1}\|_{L^2\to L^2}\leq \|(\cL(\jj,\eps)-\mu)^{-1}\|_{L^2\to L^2} \frac{M}{1-M} \,.
\end{equation} 
Therefore, using that 
$M \to 0$ as $(\jj',\eps',\mu') \to (\jj,\eps,\mu)$, we deduce the joint continuity of $(\cL(\jj,\eps)-\mu)^{-1}$ and we conclude the proof.
\end{proof}

We now prove Lipschitz continuity of the eigenfunction $H$ with respect to $\eps$, as well as a uniform lower bound on the real part of the eigenvalue with respect to $\eps$ and $\jj$.
\begin{lemma}
\label{lemma:continuity of resolvent}
Let $U \in C^\infty (\TT^3)$ be a smooth divergence free, zero average velocity field and $\eps \in (0,1]$. Assume that there exists $ \jj_\star$, a function $H_{\jj_\star}$ and a simple eigenvalue $p(\jj_\star)$ with strictly positive real part so that 
\begin{equation*}
\cL(\jj_\star,1)H_{ \jj_\star} =p(\jj_\star)H_{\jj_\star}.
\end{equation*} 
Then, there exists a curve $\Gamma \subset\{\Re(z)>0\}$, and $J, \eta>0$, so that for all $|\jj-\jj_\star| \leq J$, $\eps \in [1-\eta,1]$ we have that
\begin{equation*}
H(x;\jj,\eps)=\frac{1}{2\pi i}\int_{\Gamma} (\mu-\cL(\jj,\eps))^{-1}H_{ \jj_\star} \dd\mu
\end{equation*}
is an eigenfunction of $\cL(\jj,\eps)$ with simple eigenvalue $p(\jj,\eps)$ satisfying
\begin{equation*}
\Re(p(\jj,\eps))\geq \frac{1}{2}\Re(p(\jj_\star)) \,.
\end{equation*}
Furthermore, the function 
\begin{equation*}
\eps \mapsto H(x;\jj,\eps)=\frac{1}{2\pi i}\int_{\Gamma} (\mu-\cL(\jj,\eps))^{-1}H_{\jj_{\star}} \dd\mu
\end{equation*}
satisfies
\begin{equation}
\label{eq:continuityOfH}
\|H(x;\jj,\eps')-H(x;\jj,\eps)\|_{L^2(\TT^3)}\leq C |\eps' -\eps| 
\end{equation}
for any $\eps, \eps' \in [1-\eta,1]$ and $| \jj -  \jj_\star | \leq J$, where the constant $C$ is independent of $\jj,\eps,\eps'$.
\end{lemma}

\begin{proof}[Proof of Lemma \ref{lemma:continuity of resolvent}.]
We aim to apply Lemma \ref{Abstract spectral Lemma} to the operators $T_0=\cL(\jj_\star,1)$, $T=\cL(\jj,\eps)$, using the bounds from Lemma \ref{resolvent bounds}. Indeed, start by fixing some curve $\Gamma \subset \rho(\cL(\jj_\star,1))$ that is entirely contained in the half plane $\{z\in \CC:\Re(z)>\frac{1}{2}\Re(p(\jj_\star)\}$, and whose interior contains no spectral points of $\cL(\jj_\star,1)$ other than $p(\jj_\star)$. 
By property \eqref{eq:resolventDifference} of Lemma \ref{resolvent bounds} we find that for any $\mu \in \Gamma$, there holds that $$\|(\cL(\jj,\eps)-\cL(\jj_\star ,1))(\cL(\jj_\star ,1)-\mu )^{-1}\|_{L^2\to L^2}\leq C (|\jj-\jj_\star |+|\eps-1|)(1+\|(\cL(\jj_\star,1)-\mu)^{-1}\|_{L^2 \to L^2})\,,$$ where the constant $C$ depends continuously on $\jj,\eps, \mu$. 
Since the resolvent is a continuous function in $\mu$, $\|(\cL(\jj_\star,1)-\mu)^{-1}\|_{L^2\to L^2}$ attains a maximum on the compact set $\Gamma$. Hence, there holds uniformly for $\mu \in \Gamma$ that $$\|(\cL(\jj,\eps)-\cL(\jj_\star ,1))(\cL(\jj_\star ,1)-\mu )^{-1}\|_{L^2\to L^2}\leq C (|\jj-\jj_\star |+|\eps-1|),$$ where $C$ is continuous in $\jj,\eps$. 
We can thus pick $J>0, \eta>0$ so that for any $|\jj-\jj_\star|\leq J$, $|\eps-1|\leq \eta$ there holds 
\begin{equation}
\label{eq:Gamma In Resolvent}
M:=\sup_{\mu \in \Gamma} \|(\cL(\jj,\eps)-\cL(\jj_\star ,1))(\cL(\jj_\star ,1)-\mu )^{-1}\|_{L^2\to L^2} <\frac{1}{1+|\Gamma|\sup_{\mu \in \Gamma}\|\cL(\jj_\star,1)-\mu)^{-1}\|_{L^2\to L^2}},
\end{equation}
so that $\Gamma \subset \rho(\cL(\jj,\eps))$ for any $|\jj-\jj_\star|\leq J$, $|\eps-1|\leq \eta$.
Furthermore, by Lemma \ref{Abstract spectral Lemma}, this bound implies that the operator $\cL(\jj,\eps)$ has at least one eigenfunction with eigenvalue $p(\jj,\eps)$ contained in the interior region confined by $\Gamma$. In particular, if the eigenvalue $p(\jj_\star,1)$ is simple, then the eigenspace corresponding to $p(\jj,\eps)$ is also simple.

Next, we shall show the continuity of the map $\eps \mapsto H(x;\jj,\eps)$. Indeed, we once again apply Lemma \ref{Abstract spectral Lemma} with $T_0=\cL(\jj,\eps)$, $T=\cL(\jj,\eps')$. By Lemma \ref{resolvent bounds}, there exists a constant $C(\jj,\eps,\eps',\mu)$ depending continuously on its arguments (as long as $\mu\in \rho(\cL(\jj,\eps))$) so that it holds 
\begin{equation}
\|(\cL(\jj,\eps)-\cL(\jj,\eps'))(\mu-\cL(\jj,\eps))^{-1}\|_{L^2\to L^2}\leq C(\jj,\eps,\eps',\mu)(|\eps-\eps'|).
\end{equation}
But note that by \eqref{eq:Gamma In Resolvent}, for any $|\jj-\jj_\star|\leq 1$, $|\eps-1|\leq \eta$ it holds that $\Gamma \subset \rho(\cL(\jj,\eps))$. Therefore, $C(\jj,\eps,\eps',\mu)$ is a continuous function on a compact set, and thus attains a maximum on it.
Therefore, by the estimate \eqref{Riesz estimate}, it holds true for any $|\jj-\jj_\star|\leq J$, $|\eps-1|\leq \eta$, $|\eps'-1|\leq \eta$ that
\begin{equation}
\|H(\cdot;\jj,\eps)-H(\cdot;\jj,\eps')\|_{L^2(\mathbb{T}^3)}\leq C \sup_{\mu \in \Gamma}\|(\cL(\jj,\eps)-\mu)^{-1}\|_{L^2 \to L^2}|\eps-\eps'|\|H_{\jj_\star}\|_{L^2(\mathbb{T}^3)}.
\end{equation}
By Lemma \ref{resolvent bounds} we know that $\|(\cL(\jj,\eps)-\mu)^{-1}\|_{L^2 \to L^2}$ is also jointly continuous in its arguments, and therefore we may bound it uniformly for $(\jj,\eps,\mu) \in [\jj_\star-J,\jj_\star+J]\times [1-\eta,1+\eta]\times \Gamma$,
which completes the proof.

\end{proof}

\subsection{A Bloch theorem and extension to the whole space} \label{subsec:bloch}
Under the assumptions of Proposition \ref{prop:growing}, the function $B(t,x)=H(x)\e^{i\jj\cdot x+pt}$ is an exponentially growing plane wave modulated by a $2\pi$-periodic function $H$, which solves the passive vector equation \eqref{passive-vector} with $\eps=1$ and $u=U$. In general, $B$ is not itself $2\pi$-periodic, and therefore the rescaling in \eqref{def:ansatz} does not provide an example of fast dynamo on $\TT^3$. However, we can use Proposition \ref{prop:growing} to construct a proper, finite-energy solution to \eqref{passive-vector} on $\RR^3$. This is in the spirit of a Bloch type results in quantum physics \cite{tao2023pde}. The result reads as follows. 

\begin{lemma}
\label{Lemma: Bloch1}
Let $\JJ = \TT^3$, $G \in L^2 (\TT^3 \times \JJ)$ and define 
\begin{equation}
F(x)=\int_{\JJ} G(x;\jj)\e^{i \jj \cdot x}\dd \jj.
\end{equation}
Then $F \in L^2 (\RR^3)$ and 
\begin{equation} 
\int_{\RR^3} |F(x)|^2 \dd x = \int_{\JJ} \int_{\TT^3} |G(x;\jj)|^2 \dd x \dd \jj .
\end{equation}
\end{lemma}

\begin{proof}
For any $\bk \in 2\pi \ZZ^3$ and $R \in \{ k : \exists n \in \NN \text{ such that } k = (2n+1) \pi \}  $  we denote by 
\begin{equation}
|\bk | = \sup_{i =1,2,3} |\bk_i| \,, \qquad Q_R = \{ \bk \in \RR^3 : |\bk| \leq R \} \,,
\end{equation}
and $\ZZ_R = Q_R \cap 2 \pi \ZZ^3$. Then
\begin{equation} 
\int_{Q_R} |F(x)|^2\dd x = \sum_{\bk \in \ZZ_R} \int_{\bk + \TT^3 } \left |\int_{\JJ} G(x;\jj) \e^{ix\cdot \jj} \dd\jj \right |^2\dd x \,.
\end{equation}
We now note that 
\begin{equation}
\int_{\bk +\TT^3} \left |\int_{\JJ} G(x;\jj) \e^{ix\cdot \jj} \dd\jj\right|^2\dd x =\int_{\bk +\TT^3} \int_{\JJ} G(x;\jj) \e^{ix\cdot \jj} \dd\jj \cdot \int_{\JJ} \overline{G(x;\jj')} \e^{-ix\cdot \jj'} \dd\jj' \dd x.
\end{equation}
Furthermore, we can change variables,  $x = y+ \bk$, and note that $G$ is $2\pi$-periodic, to get that this expression is equal to 
\begin{equation}
\int_{\bk +\TT^3} \left |\int_{\JJ} G(x;\jj) \e^{ix\cdot \jj} \dd\jj\right|^2\dd x =\int_{\TT^3} \int_{\JJ} G(x;\jj) \e^{ix\cdot \jj}  \cdot \int_{\JJ} \overline{G(x;\jj')} \e^{-ix\cdot \jj'} \e^{ i \bk \cdot (\jj-\jj')} \dd \jj' \dd \jj  \dd x
\end{equation}
Hence, reintroducing the summation in $\bk$, we get 
\begin{equation}
\int_{Q_R} |F(x)|^2\dd x = \int_{\TT^3} \int_{\JJ} G(x;\jj) \e^{ix\cdot \jj} \cdot  \int_{\JJ} \overline{G(x;\jj')} \e^{-ix\cdot \jj'} \sum_{k \in \ZZ^3_R}\e^{ i\bk \cdot (\jj- \jj')} \dd \jj' \dd \jj   \dd x.
\end{equation}
Now, for fixed $x \in \TT^3$, set $g(x;\jj')=\overline{G(x;\jj')} \e^{-ix\cdot \jj'}$. Then, note that  
\begin{equation}
\int_{\JJ}g(x;\jj') \sum_{\bk \in \ZZ^3_R} \e^{ i \bk \cdot (\jj- \jj')} \dd \jj' =(P_{|\bk| \leq R} g(x;\cdot))(\jj) \,, 
\end{equation}
where $P_{|\bk| \leq R}$ is the projection onto the first $R$ Fourier modes. Hence, we have 
\begin{equation}
\label{bloch inequality}
\int_{B_R} |F(x)|^2\dd x = \int_{\JJ} \int_{\TT^3}  G(x;\jj) \e^{ix\cdot \jj} \cdot (P_{|\bk|\leq R}g(x;\cdot))(\jj) \dd x \dd \jj \,.
\end{equation}
Finally, note that by definition of $g$ we have 
\begin{equation}
\int_{\JJ}|G(x;\jj) \e^{ix\cdot \jj} \cdot (P_{|\bk|\leq R}g(x;\cdot))(\jj)| \dd \jj \leq \int_{\JJ}|G(x;\jj)|^2 \dd \jj \in L^2(\TT^3) 
\end{equation}
so that taking the limit $R \to \infty$ in \eqref{bloch inequality}, by dominated convergence theorem and again the definition of $g$  we have 
\begin{equation}
\lim_{R\to \infty } \int_{Q_R} |F(x)|^2\dd x = \int_{\TT^3} \lim_{R \to \infty} \int_{\JJ}  G(x;\jj) \e^{ix\cdot \jj} \cdot  (P_{|\bk|\leq R}g(x;\cdot))(\jj) \dd x \dd \jj .
\end{equation}
Observing that for every fixed $x \in \TT^3$ we have  
\begin{equation}
(P_{|\bk|\leq R}g(x;\cdot))(\jj) \rightharpoonup \overline{G(x;\jj)} \e^{- ix\cdot \jj}
\end{equation}
weakly in $L^2(\mathbb{J})$ as $R \to \infty$, we conclude the proof.
\end{proof}

\subsection{Proof of Theorem \ref{thm:Section2}} \label{subsec:proof-theorem}
We now have all the tools needed to provide a proof of the main result of this section. To avoid confusion, we once again specify the dependence on $\jj,\eps$ of the operator $\cL$ in \eqref{eq:mainop} and write
\begin{equation}\label{eq:L-j}
\cL(\jj,\eps)H=(\nabla +i\jj) \times (U \times H)+ \eps(\nabla +i \jj)^2 H.
\end{equation}
\begin{proof}[Proof of Theorem \ref{thm:Section2}]
Let us start off the proof with a computation. Indeed, suppose that we have constructed a solution $f \in L^\infty_{loc}  ( 0, \infty;L^2(\RR^3))$ to the equation 
\begin{equation}
\partial_t f=\nabla \times (U\times f)+\eps \Delta f
\end{equation}
for some $\eps \in [\zeta,1]$, and some velocity field $U$. If we then set $g(t,x)=f(t,\zeta^{-\frac{n}{2}}x)$, and define $V=\zeta^\frac{n}{2}U(x\zeta^{-\frac{n}{2}})$, it holds that $g$ solves 
\begin{equation}
\label{eq:rescalingEq}
\partial_t g=\nabla \times (V \times g)+\eps \zeta^{n} \Delta g.
\end{equation}
From this, we deduce that as long as we are able to construct exponentially growing solutions to \eqref{passive-vector} with fixed velocity field $U$, and diffusivity in the \emph{compact} set $[\zeta,1]$, a simple rescaling argument allows us to construct solutions with magnetic diffusivity $\eps \zeta^n$ and velocity field $\zeta^\frac{n}{2}U(x\zeta^{-\frac{n}{2}})$. We thus aim to use this observation to reduce to study the case $\eps  \in [\zeta, 1]$.

\medskip

\emph{Step 1.} We firstly aim at proving the exponential growth of the solution $S_\eps^{U_{n_\eps}} (t) F^\eps$.
To do so, note that the assumptions of Lemma \ref{lemma:continuity of resolvent} are satisfied for any $U = \nabla \times \Psi \in W^{1,\infty}(\TT^3)$ that is alpha-unstable, thanks to Proposition \ref{prop:growing}.  Hence, by Lemma \ref{lemma:continuity of resolvent}, there exist $J, \eta>0$, so that we may define the function
\begin{equation}
H(x;\jj,\eps)=\frac{1}{2\pi i}\int_{\Gamma} (\mu-\cL(\jj,\eps))^{-1}H_{ \jj_\star} \dd\mu
\end{equation}
for $|\jj-\jj_\star|<J$, $|\eps-1|<\eta$, and 
\begin{equation}
H(x;\jj,\eps)=\frac{1}{2\pi i}\int_{\Gamma} (\mu-\cL(\jj,\eps))^{-1}\overline{H}_{ \jj_\star} \dd\mu
\end{equation}
for $|\jj+\jj_\star|<J$, $|\eps-1|<\eta$. We cliam that these are eigenfunctions of $\cL(\jj,\eps)$. Indeed, by Lemma \ref{lemma:continuity of resolvent}, as soon as we know that $\overline{H}_{\jj_\star}$ is an eigenfunction of $\cL(-\jj_{\star},\eps)$, it follows that $H(x;\jj,\eps)$ are eigenfunctions of $\cL(\jj,\eps)$ with simple eigenvalues $p(\jj,\eps)$ of real part at least $\frac{1}{2}\Re(p(\jj_\star))$. Furthermore, it holds that for $|\jj-\jj_\star|<J$, there holds 
\begin{equation}
\overline{H(x,\jj,\eps)}=H(x,-\jj,\eps), \quad \overline{p(\jj,\eps)}=p(-\jj,\eps).
\end{equation}
Indeed, to see this, note that if $\cL(\jj,\eps)H=\lambda H$, then taking complex conjugates, and noting that $\overline{\cL(\jj,\eps)}=\cL(-\jj,\eps)$, there holds $\cL(-\jj,\eps)\overline{H}=\overline{\lambda}\overline{H}$. Hence, $\overline{H}_{\jj_\star}$ is a simple eigenfunction of $\cL(-\jj_\star,\eps)$, and therefore for $|\jj+\jj_{\star}|<\eta$, $H(x,\jj,\eps)$ is still a simple eigenfunction of $\cL(\jj,\eps)$ with eigenvalue $p(\jj,\eps)=\overline{p(-\jj,\eps)}$. Therefore, we may now fix any $\zeta \in [1-\eta, 1)$, and let $\eps \in (\zeta^{n_\eps+1}, \zeta^{n_\eps}]$ for some $n_\eps \in \NN_0$. We define the Ansatz solution $B^\eps$ of \eqref{passive-vector} with velocity field  $\zeta^{n/2} U (\frac{x}{\zeta^{n/2}})$ as follows. Note first that $\frac{\eps}{\zeta^n} \in (\zeta,1]$. Thus, it holds    
$$
\e^{\cL(\jj,\frac{\eps}{\zeta^n})t}H\left(x;\jj,\frac{\eps}{\zeta^n}\right)=\e^{p(\jj,\frac{\eps}{\zeta^n}) t}H\left(x;\jj,\frac{\eps}{\zeta^n}\right)
$$ 
by our construction.
Then we set 
$$ 
B^\eps(t,x)= \frac{\zeta^{- 3n/2}}{\| H(\cdot;\cdot,\frac{\eps}{\zeta^n}) \|_{L^2 (\TT^3 \times Q_J (\jj_\star) \cup Q_J(-\jj_\star)) }} \int_{Q_J (\jj_\star) \cup Q_J(-\jj_\star)}\e^{p(\jj,\frac{\eps}{\zeta^n})t}H\left(\frac{x}{\zeta^{n/2}};\jj,\frac{\eps}{\zeta^n}\right)\e^{i\jj \zeta^{-n/2} \cdot x} \dd\jj \,. 
$$
Note that for all $t \geq 0, x \in \RR^3$, $B^\eps(t,x)$ is \emph{real valued}, since $\overline{p(\jj,\frac{\eps}{\zeta^n})}=p(-\jj, \frac{\eps}{\zeta^n})$, $\overline{H(x,\jj,\eps  \zeta^{-n})}=H(x,-\jj,\eps \zeta^{-n})$.
Since  $\e^{p(\jj,\frac{\eps}{\zeta^n})t}H(x;\jj,\frac{\eps}{\zeta^n})\e^{i\jj \cdot x}$ is a solution of \eqref{passive-vector}  with magnetic diffusivity $\eps \zeta^{-n} $ and velocity field $U$ for every $\jj$, 
it follows that 
$$ 
\int_{Q_J (\jj_\star) \cup Q_J(-\jj_\star)}\e^{p(\jj,\frac{\eps}{\zeta^n})t}H\left(x ;\jj,\frac{\eps}{\zeta^n}\right)\e^{i\jj\cdot x} \dd\jj \,
$$ 
is as well a  solution of \eqref{passive-vector} with diffusivity $\eps \zeta^{-n}$ and velocity field $U$. Hence, the computation \eqref{eq:rescalingEq} at the beginning of the proof implies that $B^\eps$ is a solution to \eqref{passive-vector} with velocity field $\zeta^{n/2} U (\frac{x}{\zeta^{n/2}})$ and diffusivity parameter $\eps$. Furthermore, Lemma \ref{lemma:continuity of resolvent} further implies that $\Re(p(\jj,\frac{\eps}{\zeta^n}) \geq \frac{1}{2}\Re(p(\jj_\star))$. Hence, by this inequality, Lemma \ref{Lemma: Bloch1} and a change of variable, it holds that
\begin{align*}
\|B^\eps(t)\|_{L^2}^2 
& = \frac{\zeta^{- 3n}}{\| H (\cdot; \cdot , \frac{\eps}{\zeta^n}) \|_{L^2  }^2} \int_{\RR^3}\bigg|\int_{Q_J (\jj_\star) \cup Q_J(-\jj_\star)}H\left(x \zeta^{-n/2};\jj,\frac{\eps}{\zeta^n}\right)\e^{i\jj \zeta^{-n/2} \cdot x} \dd\jj\bigg|^2 \dd x \\
& = \frac{1}{\| H (\cdot ; \cdot , \frac{\eps}{\zeta^n}) \|_{L^2  }^2} \int_{\RR^3}\bigg|\int_{Q_J (\jj_\star) \cup Q_J(-\jj_\star)}H\left(x ; \jj,\frac{\eps}{\zeta^n}\right)\e^{i\jj \cdot x} \e^{p(\jj,\frac{\eps}{\zeta^n})t} \dd\jj\bigg|^2 \dd x
\\
& \geq \e^{\frac{1}{2}\Re(p(\jj_\star))t} \frac{1}{\| H (\cdot ; \cdot , \frac{\eps}{\zeta^n}) \|_{L^2  }^2} \int_{Q_J(\jj_\star) \cup Q_J(-\jj_\star)}\int_{\TT^3}\bigg|H\left(x ; \jj,\frac{\eps}{\zeta^n}\right)\bigg|^2 \dd x \dd\jj \\
&=\e^{\frac{1}{2}\Re(p(\jj_\star))t} \,.
\end{align*}
Here we identified the $3$-dimensional torus in $\jj$ with $[-\pi,\pi]^3$. Since $\mathds{1}_{Q_J(\jj_\star)\cup Q_J(-\jj_\star)}H(x;\jj,\frac{\eps}{\zeta^n})$ has compact support in this region, it may be identified with a $2\pi$-periodic function upon extending it to all of $\mathbb{R}^3$, and hence we may indeed apply Lemma \ref{Lemma: Bloch1}. Applying Lemma \ref{Lemma: Bloch1} once again to deduce $\| B^{\eps} (0) \|_{L^2} =1$ we conclude the proof of the first property.

\medskip

\emph{Step 2.} We now move on to proving the second part. Note that it suffices to prove the claim for $\eps \in [\zeta,1]$. Indeed, suppose that it is proven in this case. Then, for any $\eps \in (\zeta^{n+1},\zeta^n]$, $F^\eps $ is nothing but 
\begin{equation}
F^\eps(x)=\zeta^{-\frac{3n}{2}}F^\frac{\eps}{\zeta^n}(x\zeta^{-\frac{n}{2}}).
\end{equation}
Hence, changing variables, for any $R>0$ so that $\|F^\frac{\eps}{\zeta^n}\|_{L^2(Q_R)}\geq 1-\delta$, the same holds for $F^\eps$.
Hence, let $\eps \in [\zeta,1]$. Define now a map  $I:[\zeta,1] \times \RR_+$ to $ L^2(\RR^3)$ given by
\begin{equation}
I:(\eps,R) \mapsto \int_{E_R}|F^\eps(x,0)|^2 \dd x.
\end{equation}
Note that for fixed $\eps$, $I$ is increasing in $R$. We claim that $\inf_{\eps \in [\zeta,1]}I(\eps,R) \to 1$ as $R \to \infty$. Indeed, suppose that there exists some $\eps_m, R_m$ where $R_m \to \infty$, so that $I(\eps_m,R_m)<1-\delta$ for all $m$. Then, there is a non-relabeled converging subsequence $\eps_m \to \bar \eps \in [\zeta,1]$. Pick now $\bar R$ so that $I(\bar \eps, \bar R)>1-\frac{\delta}{2}$. Then, we have 
\begin{equation*}
I(\eps_m,R_m)\geq I(\bar \eps,\bar R )-|I(\bar \eps, \bar R )-I(\eps_m,R_m)|
\end{equation*}
For $m$ large enough, $R_m \geq \bar R$, and $I(\eps_m,R_m)<I(\bar \eps,\bar R)$ by assumption. Thus, 
$$
|I(\bar \eps, \bar R)-I(\eps_m,R_m)|=I(\bar \eps,\bar R)-I(\eps_m,R_m)\leq I(\bar \eps, R_m) -I(\eps_m,R_m).
$$
Finally, we note that 
\begin{align}
|I(\eps_m,R_m)-I(\bar \eps,R_m)|
&\leq \int_{E_{R_m}}|F^{\eps_m}(x)+F^{\bar\eps}(x)||F^{\eps_m}(x)-F^{\bar\eps}(x)|\dd x \notag\\
&\leq (\|F^{\eps_m}\|_{L^2(\RR^3)}+\|F^{\bar\eps}\|_{L^2(\RR^3)})\|F^{\eps_m}-F^{\bar\eps}\|_{L^2(\RR^3)}.\label{eq:differenceIs}
\end{align}
Since the $F^\eps$ are normalized, it remains to compute 
\begin{align*}
\|F^{\eps_m}-F^{\bar\eps}\|_{L^2(\RR^3)}^2=\int_{\RR^3}\bigg|\int_{Q_J(\jj_\star)\cup Q_J(-\jj_\star)}\left[\frac{H(x;\jj,\eps_m)}{\| H (\cdot; \cdot , \eps_m) \|_{L^2}}-\frac{H(x;\jj,\bar \eps)}{\| H (\cdot; \cdot , \bar \eps) \|_{L^2}}\right]\e^{i \jj \cdot x} \dd\jj\bigg| ^2 \dd x.
\end{align*}
Using  Lemma \ref{Lemma: Bloch1}, this is nothing but 
\begin{align*}
\int_{Q_J(\jj_\star)\cup Q_J(-\jj_\star)}\int_{\TT^3}\bigg|\frac{H(x ;\jj,\eps_m)}{\| H (\cdot; \cdot , \eps_m) \|_{L^2}}-\frac{H(x ;\jj,\bar \eps)}{\| H (\cdot; \cdot , \bar \eps) \|_{L^2}}\bigg|^2 \dd x \dd\jj.
\end{align*}
However, Lemma \ref{lemma:continuity of resolvent} implies that there exists some constant $C>0$ so that 
\begin{equation}
\|H(x ;\jj,\eps_m)-H(x ;\jj,\bar \eps)\|_{L^2(\TT^3)} \leq C|\eps_m-\bar \eps|
\end{equation}
for any $\jj \in Q_{J}(\jj_\star) \cup Q_{J}(-\jj_\star)$, $\eps_m,\bar \eps \in [1-\eta,1]$. Therefore, by the reverse triangle inequality, it further holds that
\begin{equation}
|\| H (\cdot; \cdot , \eps_m) \|_{L^2(\TT^3\times Q_J(\jj_\star) \cup Q_J(-\jj_\star))}- \| H (\cdot; \cdot , \bar\eps) \|_{L^2(\TT^3\times Q_J(\jj_\star) \cup Q_J(-\jj_\star) )}|\lesssim |Q_J (\jj_\star) \cup Q_J(-\jj_\star)|^\frac{1}{2}|\eps_m -\bar\eps|.
\end{equation}
Therefore, we estimate
\begin{align*}
\int_{Q_J(\jj_\star)\cup Q_J(-\jj_\star)}\int_{\TT^3}\bigg|\frac{H(x ;\jj,\eps_m)}{\| H (\cdot; \cdot , \eps_m) \|_{L^2  }^2}&  -\frac{H(x ;\jj,\bar \eps)}{\| H (\cdot; \cdot , \bar \eps) \|_{L^2  }^2}\bigg|^2 \dd x \dd\jj 
\\
&\leq \int_{Q_{J}(\jj_\star)\cup Q_{J}(-\jj_\star)}\int_{\TT^3} \bigg|\frac{H(x ;\jj,\eps_m)-H(x ;\jj,\bar\eps)}{\| H (\cdot; \cdot , \bar \eps) \|_{L^2}}\bigg|^2 \dd x \dd\jj
\\
& \quad +\|H(\cdot ;\cdot, \eps_m)\|_{L^2}^2\left(\frac{1}{\| H (\cdot; \cdot , \eps_m) \|_{L^2  }^2}-\frac{1}{\| H (\cdot; \cdot , \bar \eps) \|_{L^2  }^2}\right)\\
& \lesssim\frac{|Q_J(\jj_\star) \cup Q_J(-\jj_\star)|}{\| H (\cdot; \cdot , \bar \eps) \|_{L^2}}|\eps_m -\bar \eps|^2.
\end{align*}
Taking $m \to \infty$, we thus see from \eqref{eq:differenceIs} that eventually 
\begin{equation}
I(\eps_m,R_m) \geq 1-\frac{\delta}{2}
\end{equation}
which yields a contradiction. 
\end{proof}

\section{Properties and energy estimates for the kinematic dynamo equation} \label{sec:energy-estimates}
This section is devoted to the main properties of the kinematic dynamo equation \eqref{passive-vector}, including well-posedness, stability and (weighted) energy estimates. Although some results hold in greater generality, we restrict ourselves to the setting in which we are given a time-independent, divergence-free streamfunction  $\psi\in C^\infty(\TT^3)$ and the associated incompressible velocity field $u=\nabla\times\psi$. When considering \eqref{passive-vector} on $\RR^3$, such functions are periodically extended to the whole space to functions that are bounded, along with their derivatives.

\subsection{Standard energy estimates and uniqueness}
When $\eps>0$, well-posedness for \eqref{passive-vector} holds in the following sense.
\begin{lemma} \label{lemma:wellposed}
There exists a unique weak solution $B\in L^\infty_{loc}(0,\infty;L^2)\cap L^2_{loc}(0,\infty;\dot{H}^1)$ to \eqref{passive-vector} with initial condition $B_{\ini}$. In fact, there holds 
the energy estimate
\begin{equation}\label{eq:energyestim}
  \| S_\eps^u(t)  B_{\ini} \|_{L^2 }^2
  +\frac{\eps}{2}\int_0^t\|S_\eps^u(\tau)  B_{\ini} \|_{\dot{H}^1 }^2\dd \tau\leq \| B_{\ini}  \|^2_{L^2}\exp\left(\frac{1}{\eps}\|u\|^2_{L^\infty} t\right),\qquad \forall t\geq 0, 
\end{equation}
the $\eps$-independent growth bound
\begin{equation}\label{eq:L2growth0}
  \| S_\eps^u(t)  B_{\ini} \|_{L^2 } \leq \| B_{\ini} \|_{L^2} \exp \left(\| \nabla u \|_{L^\infty} t\right),\qquad \forall t\geq 0, 
\end{equation}
and the continuous dependence estimate
\begin{equation}\label{eq:contdep}
  \| S_\eps^u(t)  B_{\ini,1}-S_\eps^u(t)  B_{\ini,2} \|_{L^2 } \leq \| B_{\ini,1}-  B_{\ini,2}  \|_{L^2}\exp\left(\frac{1}{2\eps}\|u\|^2_{L^\infty} t\right),\qquad \forall t\geq 0.
\end{equation}
Finally, if $u$ is smooth then the corresponding solution is smooth as well.
\end{lemma}

\begin{proof}
The existence is a standard compactness argument. 
We   now  restrict ourselves to providing formal estimates, which can be justified in a suitable approximation scheme. Also, it is clear that \eqref{eq:contdep} follows from \eqref{eq:energyestim} by linearity.

Let $B\in L^\infty_tL^2_x\cap L^2_t\dot{H}^1_x$ be a weak solution to \eqref{passive-vector}, with initial datum $B_{\ini}$. Testing the equation with $B$ in $L^2(\RR^3)$ gives the identity
\begin{equation}\label{eq:step0}
    \frac12\frac{\dd}{\dd t} \|B\|^2_{L^2} +\eps \|\nabla B\|^2_{L^2} = \int_{\RR^3} (B\cdot \nabla)u \cdot B.
\end{equation}
Note that \eqref{eq:L2growth0} can be derived from the above identity, by estimating the right-hand side as 
\begin{equation}
  \int_{\RR^3} (B\cdot \nabla)u \cdot B\leq \|\nabla u\|_{L^\infty}\|B\|^2_{L^2}
\end{equation}
Instead, if we integrate by parts in \eqref{eq:step0} we obtain
\begin{equation}
   \frac12\frac{\dd}{\dd t} \|B\|^2_{L^2} +\eps \|\nabla B\|^2_{L^2} =- \int_{\RR^3} B \otimes u : \nabla B \leq\|u\|_{L^\infty} \|B\|_{L^2} \|\nabla B\|_{L^2}.
\end{equation}
Thus
\begin{equation}
   \frac{\dd}{\dd t} \|B\|^2_{L^2} +\eps \|\nabla B\|^2_{L^2} \leq \frac1\eps\|u\|^2_{L^\infty} \|B\|^2_{L^2},
\end{equation}
and the first part of \eqref{eq:energyestim} follows from an application of the Grönwall lemma. Integrating now the above equation and using \eqref{eq:energyestim} we deduce that
\begin{align}
   \eps \int_0^t\|\nabla B(\tau)\|^2_{L^2} 
   &\leq \| B_{\ini}\|^2_{L^2} + \frac1\eps\|u\|^2_{L^\infty} \int_0^t \|B(\tau)\|^2_{L^2} \dd \tau 
   \leq 2\| B_{\ini}  \|^2_{L^2}\exp\left(\frac{1}{\eps}\|u\|^2_{L^\infty} t\right) ,
\end{align}
which proves \eqref{eq:energyestim}. Finally, if $u$ is smooth, we can differentiate \eqref{passive-vector} and obtain the smoothness of the solution as well.
\end{proof}

\subsection{Weighted energy estimates and stability}
We now establish some more refined estimates which will be useful in the sequel. Roughly speaking, these give quantitative control over the solution and its tails at spatial infinity. The main result reads as follows.

\begin{lemma}\label{lemma:concentrationSummary}
    Fix a positive $\varphi \in C_c^\infty (\RR^3)$ with $\|\varphi\|_{L^\infty}=1$ and $\| \nabla \varphi \|_{W^{1,\infty}} \leq 1$, and define $u^\varphi=\nabla \times(\psi\varphi)$.
    There exists a constant $C_\psi=C(\|\psi\|_{W^{2,\infty}})>0$, which depends continuously on its argument, such that 
    \begin{equation}\label{eq:tail0}
        \| \varphi S^u_\eps(t)B_{\ini}\|^2_{L^2}\leq C_\psi\left[\|\varphi B_{\ini}\|^2_{L^2}+
        \|\nabla \varphi\|_{L^\infty} \| B_{\ini} \|_{L^2}^2 \right]  \exp \left(C_\psi t\right),
    \end{equation}
    and
        \begin{equation}\label{eq:tail1}
        \| (1-\varphi) S^u_\eps(t)B_{\ini}\|^2_{L^2}\leq C_\psi\left[\|(1-\varphi) B_{\ini}\|^2_{L^2}+
        \|\nabla \varphi\|_{L^\infty} \| B_{\ini} \|_{L^2}^2 \right]  \exp \left(C_\psi t\right),
    \end{equation}  
and
\begin{equation}\label{eq:tail2}
    \| S^u_\eps (t)B_{\ini} - S^{u_\varphi}_\eps (t) B_{\ini} \|_{L^2 (\RR^3)}^2 \leq \frac{C_\psi}{\eps}\left[\|(1-\varphi) B_{\ini}\|^2_{L^2}+
        \|\nabla \varphi\|_{W^{1,\infty}} \| B_{\ini} \|_{L^2}^2 \right]  \exp \left(C_\psi t\right) ,
\end{equation}
for every $t\geq 0$.
\end{lemma}

\begin{proof}
The estimates leading to \eqref{eq:tail0} are the same as \eqref{eq:tail1}, and for \eqref{eq:tail2} are similar. For the sake of clarity we prove \eqref{eq:tail0} and \eqref{eq:tail2}. 

\medskip

\emph{Proof of \eqref{eq:tail0}.}
Let $B(t)=S^u_\eps (t)B_{\ini}$.
Thanks to Lemma \ref{lemma:wellposed}, we can multiply the equation \eqref{passive-vector} by $B \varphi^2$ and integrate to get
    \begin{equation}\label{eq:L2test}
        \frac12\frac{\dd}{\dd t}\|\varphi B\|^2_{L^2}+\eps\|\varphi\nabla B\|^2_{L^2} \leq 
        \int_{\RR^3} |u | |\nabla \varphi | |B|^2 \varphi  
        + \int_{\RR^3} |\nabla u| |\varphi B|^2  
           +  \eps  \int_{\RR^3} |\varphi \nabla B| |\nabla \varphi| |B|.
        \end{equation}
Thus, Young's inequality and the bound on $\varphi$ imply
    \begin{equation}
        \frac{\dd}{\dd t}\|\varphi B\|^2_{L^2}\leq  2\|\nabla u\|_{L^\infty}\|\varphi B\|^2_{L^2}+
        2\left(\|u\|_{L^\infty}\|\nabla \varphi\|_{L^\infty}  +\|\nabla \varphi\|_{L^\infty}^2\right)\|B\|^2_{L^2}.
    \end{equation}    
Using \eqref{eq:L2growth0}, we can integrate directly the the above inequality to get
    \begin{equation}\label{eq:L2test2}
        \|\varphi B(t)\|^2_{L^2}\leq \left[\|\varphi B_{\ini}\|^2_{L^2}+
        2t\left(\|u\|_{L^\infty}\|\nabla \varphi\|_{L^\infty} +\|\nabla \varphi\|_{L^\infty}^2\right)\| B_{\ini} \|_{L^2}^2 \right]  \exp \left(2\| \nabla u \|_{L^\infty} t\right),
    \end{equation}     
     and we conclude the proof of \eqref{eq:tail1} by the inequality $t\e^{\alpha t}\leq \alpha^{-1}\e^{2\alpha t}$.

\medskip

\emph{Proof of \eqref{eq:tail2}.}
Let $B(t)=S^u_\eps (t)B_{\ini}$ and $B_\varphi(t)=S^{u_\varphi}_\eps (t) B_{\ini}$. Their difference satisfies
\begin{equation}
\partial_t (B_\varphi - B ) = \nabla \times (u_\varphi \times (B_\varphi - B)) + \nabla \times ((u_\varphi -u) \times B) + \eps \Delta (B_\varphi - B) \,.
\end{equation}
A standard energy estimate similar to \eqref{eq:L2test}  entails
\begin{equation}
\frac12\frac{\dd}{\dd t}\|B_\varphi - B\|^2_{L^2}
+\eps \|B_\varphi - B\|^2_{\dot{H}^1}
\leq \int_{\RR^3}|\nabla u_\varphi||B_\varphi - B|^2
+\int_{\RR^3} |u_\varphi-u| | B| |\nabla (B_{\varphi} - B)|.
\end{equation}
Thus, using the identity $u_\varphi - u = u (\varphi -1 ) + \nabla \varphi \times \psi$ we deduce that
$\|\nabla u_\varphi\|_{L^\infty}\leq 2
\| \psi\|_{W^{2,\infty}}$ and
\begin{align}
\frac{\dd}{\dd t}\|B_\varphi - B\|^2_{L^2}
&\leq 2\int_{\RR^3}|\nabla u_\varphi||B_\varphi - B|^2
+\frac{1}{\eps}\int_{\RR^3} |u|^2 |(1-\varphi) B|^2
+\frac{1}{\eps}\int_{\RR^3} |\nabla\varphi\times \psi|^2 |B|^2\notag\\
&\leq 2\|\nabla u_\varphi\|_{L^\infty}\|B_\varphi - B\|^2_{L^2}
+\frac{4\|\psi\|_{W^{1,\infty}}^2}{\eps}\|((1-\varphi)+|\nabla \varphi|) B\|_{L^2}^2\notag\\
&\leq 4\| \psi\|_{W^{2,\infty}}\|B_\varphi - B\|^2_{L^2}
+\frac{4\|\psi\|_{W^{1,\infty}}^2}{\eps} \left[\|(1-\varphi) B\|_{L^2}^2+\|\nabla \varphi\|_{L^\infty}\| B\|_{L^2}^2\right].
\end{align}
Now, Lemma \ref{lemma:wellposed} and \eqref{eq:L2test2} in particular imply that
\begin{align}
    \|(1-\varphi) B\|_{L^2}^2+\|\nabla \varphi\|_{L^\infty}\| B\|_{L^2}^2
    \leq C_\psi\left[\|(1-\varphi) B_{\ini}\|^2_{L^2}+
        \|\nabla \varphi\|_{L^\infty} \| B_{\ini} \|_{L^2}^2 \right]  \exp \left(4 \| \psi\|_{W^{2,\infty}} t\right).
\end{align}
The Grönwall lemma together with the fact that 
$B_\varphi(0) - B(0)=0$ then gives
\begin{align}
\|B_\varphi(t) - B(t)\|^2_{L^2}
&\leq \frac{C_\psi}{\eps }\left[\|(1-\varphi) B_{\ini}\|^2_{L^2}+
        \|\nabla \varphi\|_{L^\infty} \| B_{\ini} \|_{L^2}^2 \right]  t\exp \left(4\| \psi\|_{ W^{2,\infty}} t\right).
\end{align}
Using that  the inequality $t\e^{\alpha t}\leq \alpha^{-1}\e^{2\alpha t}$, we conclude the proof of \eqref{eq:tail2}, and we are done.
\end{proof}

\section{Construction of the velocity field and proof of Theorem \ref{thm:main}} \label{sec:proof-thm}

To achieve the proof of Theorem \ref{thm:main}, we need to define an \emph{$\eps$-independent} velocity field $u \in W^{1, \infty} (\RR^3)$ such that there are \emph{$\eps$-independent} constants $c_0, \gamma >0$ 
 with the following property:  for any $\eps >0 $, there exists $B_{\ini}^\eps \in L^2 (\RR^3)$ satisfying
\begin{equation} \label{claim:main}
    \| S^u_{\eps} (t) B_{\ini}^\eps \|_{L^2} \geq c_0 \e^{\gamma t} \|B_{\ini }^\eps \|_{L^2}, \qquad \forall t \geq 0 .
\end{equation} 
This requires a construction involving a (periodic extension of any)  alpha-unstable velocity field  $U\in C^{\infty}(\TT^3)$ with a simple eigenvalue as given by Theorem \ref{thm:Section2}, suitably rescaled, localized and copied in the whole space. The main ideas are explained in the next section, while the proof of the main technical cornerstone (Proposition \ref{prop:main} below) is postponed to the later sections.

\subsection{Main proposition and proof of Theorem \ref{thm:main}}
Our starting point is Theorem \ref{thm:Section2}, so that we are given an  alpha-unstable  velocity field $U\in C^\infty(\TT^3)$  with a simple eigenvalue, given by a streamfunction $\Psi\in C^\infty(\TT^3)$, i.e. $U=\nabla\times\Psi$. Let $\gamma>0$ and $\zeta\in (0,1)$ be the constants devised in Theorem \ref{thm:Section2}, and set 
\begin{equation}\label{eq:UiPsii}
U_n (x) = \zeta^{n/2} U(\zeta^{-n/2}x), \qquad U_n =\nabla \times \Psi_n.
\end{equation}
In light of Theorem \ref{thm:Section2}, for any $\eps \in (\zeta^{n_\eps+1}, \zeta^{n_\eps}]$ there exists an initial condition $F^\eps\in L^2(\RR^3)$ such that

\begin{itemize}
    \item $\|F^\eps\|_{L^2(\RR^3)}=1$   and 
\begin{equation}
    \|S_\eps^{U_n} (t) F^\eps   \|_{L^2(\RR^3)} \geq  \e^{2\gamma t}, \qquad \forall t\geq0;
\end{equation}
    \item for any $\delta>0$, there exists $R>0$ such that 
\begin{equation}\label{eq:tails}
   \| F^\eps  \|_{L^2(Q_R)}^2 \geq 1-\delta  
\end{equation}
    for any $\eps \in (0,1]$.
\end{itemize}
The proof of \eqref{claim:main} not only requires a choice of $u$, but also of the initial datum $B^\eps_{\ini}$. 
The construction takes $\Psi_n$ and $F^\eps$ as the building blocks for the velocity and initial datum, respectively. We give a more precise idea of the construction of both of them below.

\medskip

\noindent \emph{\textbf{Construction of $u$.}}
The autonomous, incompressible velocity  field $u$ is defined through \eqref{eq:UiPsii} as a sum
\begin{equation}
\label{eq:AnsatzU}
u  =\sum_{n,\ell =1}^\infty\nabla \times (\Psi_n  \phi_{n,\ell}) =\sum_{n,\ell =1}^\infty u_{n,\ell}\,,
\end{equation}
where $\phi_{n, \ell}\in C^\infty_c (\RR^3)$, defined in Subsection \ref{subsection:def-cutoffs}, are suitable compactly supported cutoffs with disjoint supports with the property that there exists a ball $Q_{n,\ell} \subset \{\phi_{n,\ell}\equiv 1\}$ of sufficiently large radius.

\medskip

\noindent \emph{\textbf{Construction of $B^\eps_{\ini}$.}}
The initial datum is constructed as the sum 
\begin{equation}\label{eq:initialBeps}
B^\eps_{\ini}=\sum_{\ell=1}^\infty \ell^{-2} F^\eps_\ell 
\end{equation}
with $F^\eps_\ell$ ``concentrated'' on $Q_{n_\eps,\ell}$ and is made rigorous in Subsection \ref{subsection:def-cutoffs}.

The main task now is to properly choose $\{\phi_{n,\ell}\}_{n,\ell}, \{Q_{n,\ell}\}_{n,\ell}$. Roughly speaking, to check \eqref{claim:main} for some $\eps>0$ and some $t\geq 0$, we restrict our focus on a specific ball $Q_{n_\eps,\ell_t}$, and observe that the growth happens there. It is crucial that the constants $c_0,\gamma$ in \eqref{claim:main} are independent of the choices above. The important features of this construction are in the result below.

\begin{proposition} \label{prop:main}
There exist constants $c_0, \gamma >0$, collections of balls $\{Q_{n, \ell}\}\subset \RR^3$ and cut-offs $\{\phi_{n, \ell}\} \subset C^\infty_c (\RR^3)$ such that the velocity $u \in W^{1, \infty} (\RR^3)$ given in \eqref{eq:AnsatzU} satisfies the following properties. 

For every $\eps>0$  there exists $B^\eps_{\ini}\in L^2 (\RR^3)$ as in \eqref{eq:initialBeps}  such that if $\eps \in (\zeta^{n_\eps+1}, \zeta^{n_\eps}]$ for some $n_\eps\in\NN$ and $t \in [\ell_t, \ell_t +1)$ for some $\ell_t\in\NN$, then
    \begin{enumerate} [label=(P\arabic*), ref=(P\arabic*)]
    \item \label{1} $\displaystyle\ell_t^{-2} \|S^{U_{n_\eps}}_\eps(t)F^\eps_{\ell_t}\|_{L^2(Q_{n_\eps,\ell_t})} \geq 2 c_0 \e^{\gamma t}\|B^\eps_{\ini} \|_{L^2(\RR^3)}\,,$
    \item \label{2} $\displaystyle\ell_t^{-2} \|(S^{U_{n_\eps}}_\eps(t)-S^{u_{n_\eps,\ell_t}}_\eps(t))F^\eps_{\ell_t}\|_{L^2(Q_{n_\eps,\ell_t})} \leq \frac{c_0}{3}\|B^\eps_{\ini} \|_{L^2(\RR^3)}\,,$
    \item \label{3} $\displaystyle \ell_t^{-2} \|(S^{u_{n_\eps,\ell_t}}_\eps (t)-S^u_\eps (t))F^\eps_{\ell_t}\|_{L^2(Q_{n_\eps,\ell_t})} \leq \frac{c_0}{3}\|B^\eps_{\ini}\|_{L^2(\RR^3)}$\,,
    \item \label{4} $\displaystyle\bigg\|\sum_{\ell \neq \ell_t}\ell^{-2} S_\eps^u (t) F^\eps_{\ell}\bigg\|_{L^2(Q_{n_\eps,\ell_t})} \leq \frac{c_0}{3}\|B^\eps_{\ini} \|_{L^2(\RR^3)}$\,.
\end{enumerate}
\end{proposition}
We now prove Theorem \ref{thm:main} assuming Proposition \ref{prop:main}.

\begin{proof}[Proof of Theorem \ref{thm:main}]
Let $u \in W^{1, \infty} (\RR^3)$ be the velocity field given by Proposition \ref{prop:main} through \eqref{eq:AnsatzU}. We need to prove that there exist $c_0, \gamma >0$    such that  for any $\eps >0$ there exists $B_{\ini }^\eps \in L^2 (\RR^3)$ satisfying the growth estimate \eqref{claim:main}.

Fix $\eps>0$, choose $n_\eps \in \NN$ such that $\eps \in (\zeta^{n_\eps+1}, \zeta^{n_\eps} ]$, and let $B_{\ini}^\eps$ be the initial datum given by Proposition \ref{prop:main} via \eqref{eq:initialBeps}.  For any $t>0$, let $\ell_t \in \NN$ such that $t \in [\ell_t, \ell_t+1)$.   Then,
by the triangle inequality and linearity of $S^u_\eps (t)$ we have 
\begin{align*} 
\|S^u_\eps (t)B^\eps_{\ini}\|_{L^2(\RR^3)} & \geq \|S^u_\eps  (t)B^\eps_{\ini}\|_{L^2(Q_{n_\eps,\ell_t})} 
\\
& \geq \ell_t^{-2}\|S^u_\eps (t)F^\eps_{\ell_t}\|_{L^2(Q_{n_\eps,\ell_t})} -\bigg\|\sum_{\ell \neq \ell_t}\ell^{-2} S_\eps^u (t) F^\eps_{\ell}\bigg\|_{L^2(Q_{n_\eps,\ell_t})}
\\
& \geq \ell_t^{-2} \|S^{U_{n_\eps}}_\eps(t)F^\eps_{\ell_t}\|_{L^2(Q_{n_\eps,\ell_t})}- \ell_t^{-2} \|(S^{U_{n_\eps}}_\eps(t)-S^{u_{n_\eps,\ell_t}}_\eps(t))F^\eps_{\ell_t}\|_{L^2(Q_{n_\eps,\ell_t})}
\\
& \quad - \ell_t^{-2} \|(S^{u_{n_\eps,\ell_t}}_\eps(t)-S^u_\eps (t))F^\eps_{\ell_t}\|_{L^2(Q_{n_\eps,\ell_t})}-\bigg\|\sum_{\ell \neq \ell_t}\ell^{-2} S_\eps^u (t) F^\eps_{\ell}\bigg\|_{L^2(Q_{n_\eps,\ell_t})}.
\end{align*}
Applying Proposition  \ref{prop:main} we deduce that for any $\eps \in (\zeta^{n_\eps+1}, \zeta^{n_\eps}]$ and $t \in [\ell_t, \ell_t +1)$ it holds that
\begin{equation}
  \|S^u_\eps (t)B^\eps_{\ini}\|_{L^2(\RR^3)} 
  \geq c_0 \e^{\gamma t}\|B^\eps_{\ini} \|_{L^2(\RR^3)} \,.   
\end{equation}
Since $ t,\eps$   are arbitrary and $c_0, \gamma >0$ are independent of $\ell_t, n_\eps$ we conclude the proof.
\end{proof}

\subsection{Definition of the velocity field and the initial data} \label{subsection:def-cutoffs}
In this section, we define the velocity field and the initial data. In what follows, $\fU=\fU(\|\Psi\|_{W^{2,\infty}})\geq 10$ is a sufficiently large constant solely depending on the norms of $\Psi$, whose value will be fixed later. 

We first dwell on the construction of the sequence of cut-off functions $\phi_{n,\ell}$. For a fixed $\eps>0$, consider the initial condition $F^\eps$ given in Theorem \ref{thm:Section2}. 
For each $\ell\in \NN$,  we appeal to \eqref{eq:tails} and find $R_{n,\ell}> 1$  such that  
\begin{equation} \label{eq:condition-initial}
    R_{n,\ell}^{-1}+\| F^\eps \mathds{1}_{Q_{R_{n,\ell}}} - F^\eps \|_{L^2 (\RR^3)}^{ 2 } \leq  \frac{1}{\fU}\exp( -(n+1)- \fU(\ell+1) )  \,,
\end{equation}
uniformly in $\eps\in (0,1]$.
We then take a countable collection of compactly supported smooth functions satisfying the following properties:
\begin{subequations}\label{varphi-estimates}
    \begin{align}
        &0 \leq \phi_{n,\ell} \leq 1\,,  \qquad  
        \phi_{n,\ell} \equiv 1\quad  \text{ on } Q_{n,\ell} := Q_{2R_{n,\ell}} (x_{n,\ell})\,,\label{eq:varphi-estimatesA}
        \\
        &\dist(\supp(\phi_{n,\ell}),\supp(\phi_{n',\ell'})) \geq 2R_{n,\ell},  \qquad \text{ for all } (n,\ell) \neq (n',\ell') \,,\label{eq:varphi-estimatesB}
        \\
        &\| \nabla \phi_{n,\ell} \|_{L^\infty} + \| D^2 \phi_{n,\ell} \|_{L^\infty} \leq \frac{1}{\fU} \exp(- (n+1) - \fU (\ell+1)) \,.\label{eq:varphi-estimatesC}
    \end{align}
\end{subequations}
The existence of such collection of functions and balls is fairly easy to prove: once the radii are fixed by \eqref{eq:condition-initial}, the condition \eqref{eq:varphi-estimatesC} is imposing that the balls are well-separated, say by ten times the right-hand side of \eqref{eq:varphi-estimatesC}. It is then sufficient to fix a suitable well-separated sequence of centers $\{x_{n,\ell}\}$. Properties \eqref{eq:varphi-estimatesA}-\eqref{eq:varphi-estimatesB} follow immediately.

While we now have $u$ as in \eqref{eq:AnsatzU} at our disposal, it remains to construct the building blocks $F^\eps_\ell$ in \eqref{eq:initialBeps}.
Given $\eps>0$  and picking $n_\eps\in \NN$ such that $\eps\in (\zeta^{n_\eps+1},\zeta^{n_\eps}]$,  we define 
\begin{equation}\label{eq:Bepsindef}
F^\eps_\ell (\cdot) = F^\eps (\cdot - x_{n_\eps, \ell} ) \,.  
\end{equation}
The initial datum $B^{\eps}_{\ini}$ is then defined according to \eqref{eq:initialBeps}.

\subsection{Proof of Proposition \ref{prop:main}} 
The proof of the main Proposition \ref{prop:main} is split across several steps in the section, each one of them establishing \ref{1}-\ref{4}. We preliminarily notice that if $\fU \geq 10$, it holds that 
\begin{equation}\label{eq:Psicomparison}
    \frac{1}{C}\|\Psi\|_{W^{2,\infty}}\leq  \|\Psi_{n}\phi_{n,\ell}\|_{W^{2,\infty}},\bigg\|\sum_{n,\ell}\Psi_n \phi_{n,\ell}\bigg\|_{W^{2,\infty}}, \|\Psi_n\|_{W^{2,\infty}}  \leq C \|\Psi\|_{W^{2,\infty}},
\end{equation}
for some constant $C$ independent of  $n,\ell,\fU$.  This follows from the periodicity of $\Psi$, the fact that the rescaling \eqref{eq:UiPsii} do not increase the Lipschitz norm of $\nabla \Psi$, and the properties \eqref{varphi-estimates} of the cut-offs.

Furthermore, using that $\sum_{j=1}^\infty \frac{1}{j^2} = \frac{\pi^2}{6}$,  for any $\ell \geq 1$ and $\eps \in (0,1)$, we deduce from \eqref{eq:condition-initial} that  \begin{equation}
   \frac12 \leq \|B_{\ini}^\eps \|_{L^2(\RR^3)} \leq  2. \label{eq:ABcomparison}
\end{equation}
Indeed, by the summability of $\sum_{j \geq 1}j^{-2}$ the upper bound is elementary. To show the lower bound, note that 
\begin{equation}
\|B^\eps_{\ini}\|_{L^2(\RR^3)} \geq \|F^\eps_1\|_{L^2(Q_{n,1})}-\|\sum_{\ell \geq 2}\frac{1}{\ell^2}F^\eps_\ell\|_{L^2(Q_{n,1})}.
\end{equation}
By the condition \eqref{eq:condition-initial}, as long as $\fU$ is sufficiently large we have that $\|F^\eps_1\|_{L^2(Q_{n,1})} \geq \frac{9}{10}$. Furthermore, since $\dist(Q_{n,\ell},Q_{n,1}) \geq 2R_{n,1}$ for all $\ell \neq 1$, it holds that $\|F^\eps_\ell\|_{L^2(Q_{n,1})} \leq \|F^\eps\|_{L^2(\RR^3 \setminus Q_{R_{n,1}})} \leq \frac{1}{10}$. Hence, $\|\sum_{\ell\geq 2}\ell^{-2}F^\eps_\ell\|_{L^2(Q_{n,1})}\leq \frac{1}{10}\sum_{\ell \geq 2}\ell^{-2}<\frac{2}{5}$. Therefore, it holds that 
\begin{equation}
\|B_{\ini}^\eps\|_{L^2(\RR^3)} \geq \frac{1}{2}.
\end{equation} 
In other words, the Lipschitz (resp. $L^2$) norms of the building blocks of the velocity field (resp. initial data) are comparable. These facts will frequently be used in the following proofs below.

\begin{proof}[Proof of Proposition \ref{prop:main}]
To ease notation, we will omit some of the subscripts and write $n$ for $n_\eps$ and $\ell$ for $\ell_t$ in the proof. Also, the symbol $C_\Psi$ denotes a \emph{generic} constant depending continuously on $\|\Psi\|_{W^{2,\infty}}$, that may change within the proof.

\medskip

\emph{Proof of \ref{1}.}
Let $\tilde\varphi$ be a standard smooth cut-off such that $\tilde\varphi \equiv 1$ outside $Q_{n,\ell}$, and $\tilde\varphi \equiv 0$ on $Q_{\frac{3}{2}R_{n,\ell}}(x_{n,\ell})$, and $\|\nabla \tilde\varphi\|_{L^\infty} \leq 4 R_{n,\ell}^{-1}$. Then, we have  that
\begin{equation}\label{eq:416}
\|S^{U_n}_\eps (t)F^\eps_\ell\|_{L^2(Q_{n,\ell})} \geq \| S^{U_n}_\eps (t)F^\eps_\ell \|_{L^2 (\RR^3)} - \| \tilde\varphi S^{U_n}_\eps (t)F^\eps_\ell \|_{L^2 (\RR^3)}.
\end{equation}
In view of \eqref{eq:tail1} in Lemma \ref{lemma:concentrationSummary}, there holds 
\begin{equation}
\| \tilde\varphi S^{U_n}_\eps (t)F^\eps_\ell \|_{L^2 (\RR^3)}^2 \leq C_{\Psi_n}[\|\tilde\varphi F^\eps_\ell\|_{L^2(\RR^3)}^2+\|\nabla \tilde\varphi\|_{L^\infty}]\exp(C_{\Psi_n}t)].
\end{equation}
for some universal constants $C_{\Psi_n}$ depending only on $\|\Psi_n\|_{W^{2,\infty}}$.
However, in view of \eqref{eq:Psicomparison} and the continuous dependence of the constants on the norms of $\Psi$ as in Lemma \ref{lemma:concentrationSummary}, we may replace all the constants $C_{\Psi_n}$ by new universal constants $C_{\Psi}$ depending only on $\|\Psi\|_{W^{2,\infty}}$. Furthermore, by the choice of $R_{n,\ell}$ in \eqref{eq:condition-initial}, we have that 
\begin{equation}
   R_{n,\ell}^{-1}+\|F^\eps_\ell\|_{L^2(\RR^3 \setminus Q_{R_{n,\ell}}(x_{n,\ell}))} \leq \frac{1}{\fU}\exp(-(n+1)-\fU(\ell+1)),
\end{equation}
and since the support of $\tilde\varphi$ is contained in $\RR^3 \setminus Q_{R_{n,\ell}}(x_{n,\ell})$, it holds for any $t \in [\ell, \ell+1)$
\begin{align}
\| \tilde\varphi S^{U_n}_\eps (t)F^\eps_\ell \|_{L^2 (\RR^3)}^2 
&\leq C_{\Psi}\left[\|F^\eps_\ell\|^2_{L^2(\RR^3 \setminus Q_{R_{n,\ell}}(x_{n,\ell}))}+\frac{4}{R_{n,\ell}}\right]\exp({C_{\Psi}(\ell+1)})\notag\\
&\leq C_{\Psi}\fU^{-1}\exp((C_{\Psi}-\fU)(\ell+1)).
\end{align}
Hence, picking $\fU$ large depending only on $\|\Psi\|_{W^{2,\infty}}$, it holds by \eqref{eq:ABcomparison} that 
\begin{equation}\label{eq:421}
    \| \tilde\varphi S^{U_n}_\eps(t)F^\eps_\ell \|_{L^2(\RR^3)}  \leq \frac{1}{4} \|B_{\ini}^\eps\|_{L^2(\RR^3)} , \qquad \forall t \in [\ell,\ell+1).
\end{equation}
Furthermore, thanks to the growth estimate \eqref{eq:growthF} in Theorem \ref{thm:Section2} and \eqref{eq:growthF}
it holds that
\begin{equation}\label{eq:422}
\|S^{U_n}_\eps (t)F^\eps_\ell \|_{L^2(\RR^3)} \geq \e^{2\gamma t} \geq \frac12\e^{2\gamma t}\|B_{\ini}^\eps\|_{L^2(\RR^3)} \,.
\end{equation}
Therefore, by combining \eqref{eq:421} and \eqref{eq:422} with \eqref{eq:416}  we deduce that
\begin{equation}
\|S^{U_n}_\eps(t)F^\eps_\ell\|_{L^2(Q_{n,\ell})} 
\geq \frac12 \left(  \e^{2\gamma t} - \frac{1}{2}\right) \|B^\eps_{\ini}\|_{L^2(\RR^3)} \geq  \frac{1}{4} \e^{2\gamma t}  \|B^\eps_{\ini}\|_{L^2(\RR^3)}
\geq  \frac{ \gamma^2}{4} \ell^2  \e^{\gamma t} \|B^\eps_{\ini}\|_{L^2(\RR^3)},
\end{equation}
for all $t \in [\ell,\ell+1)$,
as we wanted.

\medskip

\emph{Proof of \ref{2}.}
The proof of \ref{2} is an application of Lemma \ref{lemma:concentrationSummary}. Indeed, note that up to translating $x_{n,\ell}$ to the origin, we can just apply \eqref{eq:tail2} in Lemma \ref{lemma:concentrationSummary} with $u=U_n$  to deduce 
\begin{equation}
\|(S^{U_n}_\eps(t) -S^{u_{n,\ell}}_\eps(t))F^\eps_\ell\|_{L^2(Q_{n,\ell})}^2\leq \frac{C_{\Psi_n}}{\eps}[\|(1-\phi_{n,\ell})F^\eps_\ell\|^2_{L^2(\RR^3)}+\|\nabla \phi_{n,\ell}\|_{W^{1,\infty}}] \exp(C_{\Psi_n}t).
\end{equation}
Once again, we may appeal to \eqref{eq:Psicomparison} to replace the constants $C_{\Psi_n}$ by constants $C_{\Psi}$ depending only on $\|\Psi\|_{W^{2,\infty}}$. Next, recall that $Q_{n,\ell} \subset \{\phi_{n,\ell} \equiv 1\}$. In particular, \eqref{eq:condition-initial} implies
\begin{equation}\label{eq:427}
\|(1-\phi_{n,\ell})F^\eps_\ell\|^2_{L^2(\RR^3)}\leq \|F^\eps_\ell\|_{L^2(\RR^3 \setminus Q_{n,\ell})}^2 
\leq \fU^{-1} \exp(-(n+1)-\fU(\ell+1)).
\end{equation}
Furthermore, by the assumption \eqref{eq:varphi-estimatesC}, we may bound 
\begin{equation}
\|\nabla \phi_{n,\ell}\|_{W^{1,\infty}} \leq \fU^{-1}\exp(-(n+1)-\fU(\ell+1)).
\end{equation}
Hence, since we assumed that $\eps \in (\zeta^{n+1},\zeta^n]$, there holds for $t \in [\ell,\ell+1]$
\begin{align*}
\|(S^{U_n}_\eps(t) -S^{u_{n,\ell}}_\eps(t))F^\eps_\ell\|_{L^2(Q_{n,\ell})}^2
\leq C_{\Psi}\fU^{-1}(\zeta e)^{-(n+1)}\exp(-(\fU-C_{\Psi})(\ell+1)).
\end{align*}
 Since  $\zeta > \frac{1}{e}$ by Theorem \ref{thm:Section2}, we obtain
\begin{equation}
\|(S^{U_n}_\eps(t) -S^{u_{n,\ell}}_\eps(t))F^\eps_\ell\|_{L^2(Q_{n,\ell})}^2 \leq C_{\Psi} \fU^{-1} \exp((C_{\Psi}-\fU)(\ell+1)).
\end{equation}
Hence, upon taking $\fU$ large enough depending only on $\|\Psi\|_{W^{2,\infty}}$, we can bound the above quantity with the help of \eqref{eq:ABcomparison} by 
\begin{equation}
\|(S^{U_n}_\eps(t)-S^{u_{n,\ell}}_\eps(t))F^\eps_\ell\|_{L^2(Q_{n,\ell})} \leq \frac{c_0}{16}  \leq \frac{c_0}{8} \| B^\eps_{\ini} \|_{L^2 (\RR^3)} ,
\end{equation}
for $t \in [\ell,\ell+1]$.

\medskip

\emph{Proof of \ref{3}.}
Recall that $u$ is given in \eqref{eq:AnsatzU}, so that we can write 
$\psi=\sum_{n',\ell'}\Psi_{n'} \phi_{n',\ell'}$ and $u=\nabla\times\psi$. 
Moreover, 
$u_{n,\ell}=\nabla \times ( \phi_{n,\ell}\Psi_{n})=\nabla \times (\tilde\phi\psi)$
where $\tilde\phi \in C_c^\infty(\RR)$ is any function so that $\supp(\tilde\phi) \cap \supp(\phi_{n,\ell'})=\emptyset$ for all $\ell'\neq \ell$, and $\tilde\phi \equiv 1$ on $\supp(\phi_{n,\ell})$. In particular, by the assumption \eqref{eq:varphi-estimatesB}, we can find such a $\tilde\phi$ that furthermore satisfies $\|\nabla \tilde\phi\|_{W^{1,
\infty}}\leq 10 R_{n,\ell}^{-1}$.
Thus, we may once again apply Lemma \ref{lemma:concentrationSummary} to obtain (replacing all constants $C_{\psi}$ by constants $C_{\Psi}$ depending only on $\|\Psi\|_{W^{2,\infty}}$ via \eqref{eq:Psicomparison}):
\begin{align}
&\|(S^{u}_\eps(t)- S^{u_{n,\ell}}_\eps(t))F^\eps_{\ell}\|_{L^2(Q_{n,\ell})}^2  
\leq \frac{C_{\Psi}}{\eps}\left[\|(1-\tilde\phi)F^\eps_{\ell}\|^2_{L^2(\RR^3)}+\|\nabla \tilde\phi\|_{W^{1,\infty}}\right]\exp(C_\Psi t).
\end{align}
As we did in \eqref{eq:427}, since $\tilde\phi \equiv 1$ on $\supp(\phi_{n,\ell})\supset Q_{n,\ell}$ as in \eqref{eq:varphi-estimatesA}, in view of \eqref{eq:condition-initial} it holds that 
\begin{equation}
\|(1-\tilde\phi)F^\eps_{\ell}\|^2_{L^2(\RR^3)}+\|\nabla \tilde\phi\|_{W^{1,\infty}}
\leq \frac{1}{\fU} \exp(-(n+1)-\fU(\ell+1)).
\end{equation}
Therefore, we see that 
\begin{equation}
\|(S^{u}_\eps(t)- S^{u_{n,\ell}}_\eps(t))F^\eps_{\ell}\|_{L^2(Q_{n,\ell})}^2  \leq C_{\Psi}\fU^{-1}(\zeta e)^{-(n+1)}\exp((C_\Psi-\fU)(\ell+1))
\end{equation}
for any $t \in [\ell,\ell+1)$. 
Hence, arguing exactly as for the previous term, we see that for all $\fU$ large enough depending only on $\|\Psi\|_{W^{2,\infty}}$, and for any time $t \in [\ell,\ell+1)$ this is bounded by 
\begin{equation}
\|(S^{u}_\eps(t)-S^{u_{n,\ell}}_\eps(t))F^\eps_{\ell}\|_{L^2(Q_{n,\ell})} \leq \frac{c_0}{8} \| B^\eps_{\ini} \|_{L^2 (\RR^3)} \,.
\end{equation}
which completes the proof.

\medskip

\emph{Proof of \ref{4}.}
Once again, this is an application of Lemma \ref{lemma:concentrationSummary}, equation \eqref{eq:tail0}. Indeed, note that 
\begin{equation}
\bigg\|\sum_{j \neq \ell} j^{-2}S^u_\eps(t)F^\eps_j\bigg\|_{L^2(Q_{n,\ell})}\leq\sum_{j \neq \ell} j^{-2}  \| \phi_{n,\ell} S^u_\eps(t)F^\eps_j\|_{L^2(\RR^3)}.
\end{equation}
From \eqref{eq:tail0} deduce that
\begin{equation}
 \| \phi_{n,\ell} S^u_\eps(t)F^\eps_j\|_{L^2(\RR^3)}^2 
 \leq C_\Psi\left[\|\phi_{n,\ell}  F^\eps_j\|^2_{L^2(\RR^3)}+
        \|\nabla \phi_{n,\ell} \|_{L^\infty}\right]  \exp \left(C_\Psi t\right).
\end{equation}
By definition of $F^\eps_j$ we now observe that $(F^\eps_j \phi_{n, \ell} )(\cdot + x_{n, j}) =F^\eps (\cdot) \phi_{n, \ell} (\cdot + x_{n,j} ) $. We also observe that $\phi_{n,\ell} (\cdot + x_{n, j}) \equiv 0$  on $Q_{R_{n, \ell}}$ for any $j \neq \ell$ by condition \eqref{eq:varphi-estimatesB}. Therefore, using \eqref{eq:condition-initial} we obtain that for $j \neq \ell$ it holds 
\begin{equation}
\|\phi_{n,\ell} F^\eps_j\|_{L^2(\RR^3)}^2 \leq \|F^\eps\|_{L^2(\RR^3\setminus Q_{R_{n} , \ell })}^2 \leq  \frac{1}{\fU}\exp( - \fU(\ell+1) ).
\end{equation}
Hence, for any time $t \in [\ell,\ell+1)$,  by the properties of the $\phi_{n,\ell}$ in \eqref{varphi-estimates} we deduce
\begin{equation}
\bigg\|\sum_{j \neq \ell} j^{-2}S^u_\eps(t)F^\eps_j\bigg\|_{L^2(Q_{n,\ell})} \leq (C_\Psi \fU^{-1})^{1/2}\sum_{j}j^{-2} \exp((C_\Psi-\fU)(\ell+1)/2)
\end{equation}
Thanks to \eqref{eq:ABcomparison}, we have that $\| B_{\ini}^\eps \|_{L^2 (\RR^3)} \geq \frac{1}{2}$. By possibly increasing that value of $\fU$, we therefore conclude that
\begin{equation}
\bigg\|\sum_{j \neq \ell} j^{-2}S^u_\eps(t)F^\eps_j\bigg\|_{L^2(Q_{n,\ell})}  \leq 
\frac{c_0}{8}\|B_{\ini}^\eps\|_{L^2(\RR^3)}
\end{equation}
for $t \in [\ell,\ell+1)$, completing the proof.
\end{proof}

\appendix

\section{Perturbation theory and a theorem of Kato} \label{appendix}

This appendix is a self-contained exposition on the elements of spectral and perturbation theory needed throughout the manuscript, following the influential monograph of T. Kato \cite{K76}, where these results are presented in full detail. We include proofs of key results that are central to our analysis in the main text, allowing us to simplify their statements for more direct applicability to the situations of interest.

Throughout this section, we will be working with a closed operator $(T,D(T))$ on a separable, complex Banach space $(X,\|\cdot\|)$, with domain $D(T)\subset X$. The resolvent set $\rho(T)$ of $T$ is defined as 
\begin{equation}
\rho(T):=\{\lambda \in \CC \vert (\lambda -T) :D(T) \to X \ \text{has a bounded inverse}\}
\end{equation}
and the spectrum $\sigma (T)$ of $T$ is the set
\begin{equation}
\sigma(T)=\CC\setminus \rho(T).
\end{equation}
It is a general fact that $\sigma(T)$ is a closed subset of $\CC$. If $\lambda \in \sigma(T)$, and there exists $x \in X\setminus \{0\}$ so that $\lambda x-Tx=0$, we say that $\lambda $ is an eigenvalue of $T$, and write $\lambda \in \sigma_p(T)$, the point spectrum of $T$.
We will be particular interested in the case of an \emph{isolated point of the spectrum} of $T$, i.e. the case where $\lambda \in \sigma(T)$, and there exists an open set $U \ni \lambda $ so that $U \cap \sigma(T)=\{\lambda \}$. In this case, we may associate to $\lambda$ an eigenprojection, known as the \emph{Riesz projector}. Let $\Gamma$ be a simple, closed curve in the complex plane surrounding $\lambda$, and so that $\Gamma \subset U$, then, we set
\begin{equation}
P_{\lambda} =\frac{1}{2\pi i} \int_{\Gamma} (\mu-T)^{-1} \dd\mu \,.
\end{equation}
We outline a number of the properties of this operator in the following lemma.
\begin{lemma}
\label{Riesz Projector}
The bounded operator $P_\lambda$ satisfies the following:
\begin{enumerate}
    \item $P_\lambda ^2 =P_\lambda$, i.e. it is a projection operator.
    \item T commutes with $P_\lambda$, and $T|_{\Ran(P_\lambda)}$ is a bounded operator that maps the range $\Ran(P_\lambda)$ into itself.
    \item $\sigma (T|_{\Ran(P_\lambda)})=\{\lambda\}$.
    \item If $\Ran(P_\lambda)$ is finite dimensional, then $\lambda$ is an eigenvalue of $T$, and $\Ran(P_\lambda)$ coincides precisely with the generalized eigenspace of $T$ at $\lambda$, i.e. $\Ran(P_\lambda)=\bigcup_{n \in \NN}\Ker(T-\lambda)^n$.
    \item If there exists $\mu \in \rho(T)$ so that $(T-\mu)^{-1}$ is a compact operator, then $\Ran(P_\lambda)$ is finite dimensional.
\end{enumerate}
\end{lemma}
The proof of the above Lemma \ref{Riesz Projector} may be found throughout \cite{K76}*{Chapter III, Section 5,6,7}. Using the Riesz projector, we can now say that $\lambda \in \sigma(T)$ is a \emph{simple eigenvalue} if ${\rm dim}(P_\lambda)=1$, and a \emph{semisimple} eigenvalue, if the range of $P_\lambda$ consists entirely of eigenfunctions of $T$, or equivalently  $(T-\lambda)P_\lambda =0 $.

\subsection{Riesz projectors and perturbation theory}
We now discuss how the Riesz projector can be used in our perturbative arguments. In particular, we will primarily focus on the relatively non-degenerate case, where the perturbation is, in some sense, ``at least as regular'' as the unperturbed operator. 

Given two unbounded operators $(T_0,D(T_0))$, $(T_1,D(T_1))$, we  say that $T_1$ is \emph{relatively bounded} with respect to $T_0$ if $D(T_0) \subset D(T_1)$, and there exist constants $a,b>0$ so that for all $x \in D(T_0)$, it holds 
\begin{equation}
\|T_1x\|\leq a\|x\|+b\|T_0x\|.
\end{equation}
The greatest lower bound $b_0$ of all possible values of $b$ in this definition is called the \emph{$T_0$-bound of $T_1$}.
This relative boundedness assumption significantly simplifies many technical difficulties that typically arise when dealing with perturbations of unbounded operators. In particular, it allows us to avoid complications related to the domains of the operators involved.

The following result, found in \cite{K76}*{Chapter III, Section 4},  is a fundamental consequence of this framework.
\begin{theorem}
Let $(T_0,D(T_0))$, $(T_1,D(T_1))$ be as above. Assume that 
\begin{enumerate}
    \item $(T_0,D(T_0))$ is a closed operator.
    \item $T_1$ is relatively bounded with respect to $T_0$, with $T_0$ bound strictly less than $1$. 
\end{enumerate}
Then, the operator $(T_0+T_1, D(T_0))$ is closed.
\end{theorem}
Furthermore, under the relative boundedness assumption, it can be shown that for any $\mu \in \rho(T_0)$, the operator $T_1(\mu-T_0)^{-1}:X \to X$ is bounded.
With this result in place, we are now ready to state and prove a key lemma that will serve as the main tool for our perturbative arguments.

\begin{lemma}
\label{Abstract spectral Lemma}
Let $(T_0,D(T_0))$ be a closed operator, with an isolated, discrete eigenvalue $\lambda$, and let $\Gamma$ be a simple closed contour in the complex plane such that $\lambda$ is the only spectral point of $T_0$ in its interior.
Assume that $(T_1,D(T_1))$ is relatively bounded with respect to $T_0$, with $T_0$-bound strictly less than $1$, and that 
\begin{equation}\label{eq:Mbound}
\sup_{\mu \in \Gamma}\|T_1(T_0-\mu)^{-1}\| =M<1. 
\end{equation}
Define the Riesz projections 
\begin{equation}
P_0=\frac{1}{2\pi i}\int_{\Gamma}(\mu-T_0)^{-1}\dd\mu, \quad P_1=\frac{1}{2\pi i}\int_{\Gamma}(\mu-T)^{-1}\dd\mu
\end{equation}
where $T=T_0+T_1$. Then  
\begin{equation}\label{Riesz estimate}
\|P_0-P_1\|\leq \frac{1}{2\pi}|\Gamma| \frac{M}{1-M} \sup_{\mu \in \Gamma}\|(T_0-\mu)^{-1}\|
\end{equation}
In particular, if 
\begin{equation}
M<\frac{1}{1+|\Gamma|\sup_{\mu \in \Gamma}\|(T_0-\mu)^{-1}\|}, 
\end{equation}
then $T$ has an isolated eigenvalue within $\Gamma$, and ${\rm dim}(\Ran(P_0))={\rm dim}(\Ran(P_1))$. Moreover, if $\lambda$ is a simple eigenvalue of $T_0$, then $T$ also only has a simple eigenvalue in the interior of $\Gamma$.
\end{lemma}

\begin{proof}
 We begin by noting that 
\begin{equation}
T-\mu=\left(\Id +T_1(T_0-\mu)^{-1}\right)(T_0-\mu).
\end{equation}
In view of \eqref{eq:Mbound}, $\Id +T_1(T_0-\mu)^{-1}$ is invertible and therefore 
\begin{equation}
(T-\mu)^{-1}=(T_0-\mu)^{-1}\left(\Id +T_1(T_0-\mu)^{-1}\right)^{-1}.
\end{equation}
Writing the corresponding Neumann series, we find
\begin{equation}
(T-\mu)^{-1}=(T_0-\mu)^{-1}\sum_{n \geq 0}(-1)^n\left(T_1(T_0-\mu)^{-1}\right)^n.
\end{equation}
From here, we note that for $\mu \in \Gamma$, 
\begin{align}
\left\|(T-\mu)^{-1}-(T_0-\mu)^{-1}\right\|
&\leq \left\|(T_0-\mu)^{-1}\right\| \bigg\| \sum_{n \geq 1}(T_1(T_0-\mu)^{-1})^n  \bigg\| \notag\\
&\leq \frac{M}{1-M} \sup_{\mu \in \Gamma}\left\|(T_0-\mu)^{-1}\right\|.
\end{align}
In particular, \eqref{Riesz estimate} follows.  Next, assume that $M$ is small enough so that $\|P_1-P_0\|<1$. Since $\lambda$ is a discrete spectral point of $T_0$, it follows that ${\rm dim}(\Ran(P_0))<\infty$. We claim this implies that ${\rm dim}(\Ran(P_1))<\infty$ as well. Consider the map 
\begin{equation}
P_{0}:\Ran(P_1)\to \Ran(P_{0}).
\end{equation}
Since $P_{0}$ is a projection, its range $\Ran(P_0)$ is a closed, finite-dimensional subspace. If $P_{0}$ is injective on $\Ran(P_1)$,  then it is invertible on its range by the open mapping theorem, ensuring that ${\rm dim}(\Ran(P_1))\leq {\rm dim}(\Ran(P_0))$.

To prove injectivity, assume there exists $x \in \Ran(P_1)$ such that $P_{0}x=0$. Since $P_1$ is a projection, we have $x=P_1x$, and thus
\begin{equation}
\|x\|=\|P_{0}x-P_1x\|\leq \|P_{0}-P_1\|\|x\|<\|x\|
\end{equation}
which is a contradiction unless $x=0$. Therefore, $P_0$ is injective on $\Ran(P_1)$, establishing that ${\rm dim}(\Ran(P_1))\leq {\rm dim}(\Ran(P_0))$.
Repeating the argument with $P_0$ and $P_1$ interchanged shows the reverse inequality  ${\rm dim}(\Ran(P_0))\leq {\rm dim}(\Ran(P_1))$, and thus ${\rm dim}(\Ran(P_1))= {\rm dim}(\Ran(P_0))$.

Since $P_1$ is a projection onto a finite-dimensional space, $T$ restricted to $\Ran(P_1)$ is a finite-dimensional operator (i.e.  a matrix), meaning all its spectral points are eigenvalues. This completes the proof.
\end{proof}

Finally, we state and prove a simplified version of a result in perturbation theory, tailored to obtaining unstable eigenvalues for the dynamo problem. In \cite{K76}*{Chapter VIII, Theorem 2.6}, the following result is stated under more general assumptions, requiring several additional definitions that we prefer to avoid for clarity. Instead, we present a streamlined version sufficient for our purposes.

\begin{theorem}
\label{Kato}
Let $(T_0, D(T_0))$ be a closed operator on a Banach space $X$, and let $(T_1,D(T_1))$ be another closed operator that is bounded relative to $T_0$. For sufficiently small $\kappa$, define the operator $T(\kappa)=(T_0+\kappa T_1, D(T_0))$, which is a closed operator on $X$. Let $\lambda$ be an isolated, semisimple eigenvalue of $T_0$ with finite multiplicity $m$, and let 
$P_\lambda$ be the Riesz projector onto $\lambda$. Then, the eigenvalues $\mu_j(\kappa)$ of $T(\kappa)$ near $\lambda$ satisfy the asymptotic bound 
\begin{equation}
\mu_j(\kappa)=\lambda +\kappa \mu_j+o(\kappa), \qquad j\in\{1,\dots m\}
\end{equation}
where the $\mu_j$'s are the (repeated) eigenvalues of the operator $P_\lambda T_1 P_\lambda$ restricted to $\Ran(P_\lambda)$.
\end{theorem}
The proof is based on the following result from finite-dimensional perturbation theory, which can be found in \cite{K76}*{Chapter II, Theorem 5.4}. For further details on the $\lambda$-group eigenvalues, which correspond to the eigenvalues of $T(\kappa)$ generated by the splitting of the eigenvalue $\lambda$ at $\kappa =0$, we refer to \cite{K76}*{Chapter II, Section 2}.

\begin{theorem}
\label{Thm:Kato_finite_dim}
Let $T(\kappa)$ be a family of linear operators on a finite dimensional space $M$, differentiable at $\kappa=0$. If $\lambda$ is a semisimple eigenvalue of $T(0)$, then the $\lambda$-group eigenvalues $\mu_j$ of $T(\kappa)$ are differentiable at $\kappa=0$ and satisfy the asymptotic expansion
\begin{equation}
\mu_j(\kappa)=\lambda +\kappa \mu_j^1+o(\kappa)
\end{equation}
where the $\mu_j^1$ are the repeated eigenvalues of the operator $P_\lambda T'(0)P_\lambda$ restricted to $\Ran(P_\lambda)$, and $P_\lambda$ is the Riesz projector associated with $\lambda$.
\end{theorem}

We may now provide a proof of Theorem \ref{Kato}, relying on Theorem \ref{Thm:Kato_finite_dim}.

\begin{proof}[Proof of Theorem \ref{Kato}]
Let $\Gamma$ be a curve around the eigenvalue $\lambda$, and let $P_\lambda, P_\lambda(\kappa)$ denote the Riesz projectors associated with the operators $T_0, T(\kappa)=T_0+\kappa T_1$, respectively, around $\Gamma$. Denote the resolvents of $T_0$ and $T(\kappa)$ by $R(\mu;T_0) = (T_0 - \mu)^{-1}$ and $R(\mu;T(\kappa)) = (T(\kappa) - \mu)^{-1}$, respectively. We begin by considering the operator 
\begin{equation}
V(\kappa)=I-P_\lambda+P_\lambda(\kappa)P_\lambda.
\end{equation}
Note that $V(\kappa)P_\lambda=P_\lambda(\kappa)P_\lambda$ and $V(\kappa)(I-P_\lambda)=I-P_\lambda$. By Lemma \ref{Abstract spectral Lemma}, there exists some constant $c>0$ such that $\|P_\lambda(\kappa)-P_\lambda\| \leq c \kappa$, for $\kappa$ sufficiently. Consequently, for small $\kappa,$ $V(\kappa)$ is invertible, and its inverse can be written as 
\begin{equation}
V(\kappa)^{-1}=I+(P_\lambda-P_\lambda(\kappa))P_\lambda+o(\kappa)
\end{equation}
where  $o(\kappa)$ denotes any operator with norm of order $o(\kappa)$ as $\kappa \to 0$. In particular, $V(\kappa)$ is an isomorphism from $\Ran(P_\lambda) \to \Ran(P_\lambda(\kappa))$. 

Next, consider the operator 
\begin{equation}
R_1(\mu;\kappa):=V(\kappa)^{-1}R(\mu;T(\kappa))V(\kappa)P_\lambda, \qquad \text{for } \mu\in \Gamma.
\end{equation}
This operator annihilates $\Ran(I-P_\lambda)$. Since $V(\kappa)$ maps $\Ran(P_\lambda)$ into $\Ran(P_\lambda(\kappa))$  and  $R(\mu;T(\kappa))$ commutes with $P_\lambda(\kappa)$ by functional calculus, we conclude that 
$R_1(\mu;\kappa)$ maps $\Ran(P_\lambda(\kappa))$ into itself. Furthermore, by the invertibility of $V(\kappa)$, we have $R_1(\mu;\kappa)=P_\lambda R_1(\mu;\kappa)$.
We  now expand $R_1(\mu;\kappa)$ as
\begin{equation}
R_1(\mu;\kappa)=(I+(P_\lambda-P_\lambda(\kappa))P_\lambda+o(\kappa))(R(\mu;T_0)-\kappa R(\mu;T_0) T_1 R(\mu;T_0)+o(\kappa))(I-(P_\lambda-P_\lambda(\kappa))P_\lambda). 
\end{equation}
This simplifies to 
\begin{equation} \label{eq:R-1-mu-kappa}
R_1(\mu;\kappa)=R(\mu;T_0)-\kappa R(\mu;T_0)T_1R(\mu;T_0)-R(\mu;T_0)(P_\lambda-P_\lambda(\kappa))P_\lambda+(P_\lambda-P_\lambda(\kappa))P_\lambda R(\mu;T_0)+o(\kappa). 
\end{equation}
Using the fact that $R(\mu;T_0) P_\lambda = (\lambda - \mu)^{-1} P_\lambda$ 
due to the semisimplicity of  $\lambda$ for $T_0$, we get 
\begin{equation}
P_\lambda R(\mu;T_0)(P_\lambda-P_\lambda(\kappa))P_\lambda=  P_\lambda (P_\lambda-P_\lambda(\kappa))P_\lambda R(\mu;T_0) P_\lambda.
\end{equation}
Thus, using that  $R_1(\mu;\kappa)=P_\lambda R_1(\mu;\kappa)P_\lambda$ and again $R(\mu;T_0) P_\lambda = (\lambda - \mu)^{-1} P_\lambda$, we deduce from  \eqref{eq:R-1-mu-kappa} that
\begin{align*}
R_1(\mu;\kappa)=(\lambda -\mu)^{-1}P_\lambda-\kappa (\lambda-\mu)^{-2}P_\lambda T_1P_\lambda+o(\kappa) \,.
\end{align*}
Multiplying by $\frac{-\mu}{2\pi i}$ and integrating around the contour $\Gamma$, we obtain the expression
\begin{equation}
\label{eq:perturbative expression}
V(\kappa)^{-1}T(\kappa)P_\lambda(\kappa) V(\kappa)P_\lambda=\lambda P_\lambda+\kappa P_\lambda T_1P_\lambda +o(\kappa)
\end{equation}
where we have used the fact that 
\begin{equation}
T(\kappa)P_\lambda (\kappa)=-\frac{1}{2\pi i}\int_{\Gamma}R(\mu;T_0)\mu \dd\mu \,.
\end{equation}
Observe that the $\mu_j(\kappa)'s$ are the eigenvalues of $T(\kappa)P_\lambda (\kappa)$ in the $m$-dimensional space $\Ran(P_\lambda (\kappa))$. By similarity, they are also the eigenvalues of $V(\kappa)^{-1}T(\kappa)P_\lambda (\kappa)V(\kappa)$ in the $m$-dimensional space $\Ran(P_\lambda)$. Since $P_\lambda $ acts as the identity on $\Ran(P_\lambda)$, they are also equal to the eigenvalues of $V(\kappa)^{-1}T(\kappa)P_\lambda (\kappa)V(\kappa)P_\lambda $ on $\Ran(P_\lambda)$. 
Finally, applying Theorem \ref{Thm:Kato_finite_dim} and using the asymptotic expression \eqref{eq:perturbative expression}, the proof follows immediately.
\end{proof}

\addtocontents{toc}{\protect\setcounter{tocdepth}{0}}
\section*{Acknowledgments}
The research of MCZ was partially supported by the Royal Society URF\textbackslash R1\textbackslash 191492 and and the ERC/EPSRC Horizon Europe Guarantee EP/X020886/1. MS acknowledges support from the Chapman Fellowship at Imperial College London. The research of DV was funded by the Imperial College President’s PhD Scholarships.

\addtocontents{toc}{\protect\setcounter{tocdepth}{1}}

 \bibliographystyle{abbrv}
 \bibliography{DynamoBiblio.bib}

\end{document}